\documentclass{amsart} 

\usepackage{amsmath}
\usepackage{epsfig}
\usepackage{color,graphics}
\usepackage{booktabs}
\usepackage{epstopdf}
\usepackage[all]{xy}
\usepackage{fixmath, bm, amssymb, amsfonts, latexsym}

\newcommand{\ID}{{\rm I\hskip-0.02in D}}

\newtheorem{theorem}{Theorem}[section]
\newtheorem{proposition}[theorem]{Proposition}
\newtheorem{lemma}[theorem]{Lemma}
\newtheorem{corollary}[theorem]{Corollary}

\newtheorem{example}[theorem]{Example}
\newtheorem{remark}[theorem]{Remark}

\newcommand{\baa}{\begin{eqnarray*}}
\newcommand{\eaa}{\end{eqnarray*}}
\newcommand{\ba}{\begin{equation}}
\newcommand{\ea}{\end{equation}}
\newcommand{\ZZ} {\mathbb Z}
\newcommand{\RR}{\mathbb R}
\newcommand{\NN} {\mathbb N}

\newcommand{\BB} {\mathbb B}

\def\xj{x_{j+1/2}}
\def\CE{\mathcal E}
\def\HRule{\rule{\linewidth}{0.1mm}}

\def\bau{{\bar u}}
\def\bax{{\bar x}}

\def\hf{{\hat {f}}}

\def\mO{{\mathcal O}}
\def\mS{{\mathcal S}}

\def\bp{{\bf p}}

\def\br{{\bf r}}

\def\bD{{\bf D}}

\def\bL{{\bf L}}

\def\bR{{\bf R}}
\def\bT{{\bf T}}
\def\bU{{\bf U}}
\def\bV{{\bf V}}
\def\bW{{\bf W}}

\def\vp{\varphi}

\begin{document}
\author{Youngsoo Ha $^{\S}$, Chang Ho Kim$^\natural$, Hyoseon Yang ${}^\dagger$, and Jungho Yoon ${}^\ddagger$}
\title{Improving accuracy of the fifth-order WENO scheme by using the exponential approximation space
}
\thanks{
$^{\S}$ Dept. of Math. Sciences, Seoul National University, Seoul,
S. Korea ({youngamath@snu.ac.kr}), $^\natural$
Dept. of Software Technology, Glocal Campus, Konkuk University,
Chungju, S. Korea (kimchang@kku.ac.kr),
${}^\dagger$ Dept. of Computational Mathematics, Science and Engineering, Michigan State University, East Lansing, MI, USA (hyoseon@msu.edu) , ${}^\ddagger$
Dept. of Math., Ewha Womans University, Seoul, S. Korea (yoon@ewha.ac.kr).}
\date{\today}
 \maketitle


\begin{abstract}
The aim of this study is to develop a novel WENO scheme that improves
the performance of the well-known fifth-order WENO methods.
The approximation space consists of exponential polynomials with
a tension parameter that may be optimized to fit the
the specific feature of the data,
yielding better results compared to the polynomial approximation space.
However, finding an optimal tension parameter is a very important
and difficult problem, indeed a topic of active research.
In this regard, this study introduces a practical approach
to determine an optimal tension parameter by taking into account
the relationship between the tension parameter and the accuracy of
the exponential polynomial interpolation under the setting
of the fifth-order WENO scheme.
As a result, the proposed WENO scheme attains an improved order
of accuracy (that is, sixth-order) better than other fifth-order WENO methods
without loss of accuracy at critical points.
A detailed analysis is provided to verify the improved convergence rate.
Further, we present modified nonlinear weights
based on $L^1$-norm approach along with a new global smoothness indicator.
The proposed nonlinear weights reduce numerical dissipation significantly,
while attaining better resolution in smooth regions.
Some experimental results for various benchmark test problems
are presented to demonstrate the ability
of the new scheme.
\end{abstract}

\bigskip
\noindent 2000 AMS(MOS) Classification:
41A05, 41A10, 42A10, 65M06, 65M15


\pagestyle{myheadings} \thispagestyle{plain} \markboth{Y. HA, C.KIM,
H. YANG  AND J. YOON}{Improving accuracy of the fifth-WENO scheme}

\section{Introduction}

Hyperbolic systems are used for a wide range of scientific and engineering
applications such as meteorology, gas dynamics,
shallow water modeling, astrophysics models, and multiphase flow problems.
It is well-known that the hyperbolic conservation laws may generate
discontinuities in its solution even though the initial condition is smooth.
Such discontinuities introduce undesirable artifacts
like spurious oscillations in the numerical solutions.
To avoid such phenomena, Total-Variation Diminishing (TVD)
techniques have been developed \cite{Har0, Har2},
but these schemes were revealed to have  at most first-order accuracy.
To overcome of this limitation, a series of essentially non-oscillatory (ENO)
schemes have been developed.
The ENO schemes \cite{Har_osh2, Har4, Har5, shu1, shu2}
are designed to utilize several candidate stencils to avoid cross-shock
interpolation such that they  reduce spurious oscillations
near discontinuities  while achieving high order accuracy on smooth areas.
The main idea of the weighted ENO (WENO) technique is to use
a convex combination of all the candidate stencils of ENO in
a nonlinear fashion and assigns a weight to each local solution
based on its smoothness.

In  \cite{L2}, Liu et al. developed a weighted ENO scheme
of a finite volume version which had the $(r+1)$-th order accuracy
from the $r$th order ENO scheme on smooth regions
using interpolating functions obtained from all candidate stencils
in the ENO method.
Later, Jiang and Shu \cite{Jiang1} (called classical WENO or WENO-JS)
introduced new smoothness indicators that measure the regularities
of local solutions with $L^2$-norm
to obtain the fifth-order accuracy on smooth regions.
Although the WENO-JS possesses the fifth-order convergence rate
in smooth regions,
Henrick et al. noticed \cite{hen}
that it suffers loss of accuracy near the
critical points where the first
and third derivatives do not vanish simultaneously.
To correct this deficiency, the mapped WENO (hereafter, called WENO-M)
scheme was devised in the form of a mapping function on the WENO-JS weights,
leading to the maximal rate of convergence while
achieving improved results near discontinuities \cite{hen}.
Subsequently, Borges et al. \cite{bor} proposed another version of
WENO schemes (called WENO-Z) by adding a new high order reference
smoothness indicator consisting of a linear combination of
the original smoothness indicator of WENO-JS.
The WENO-M and WENO-Z schemes possess good shock capturing abilities,
but both schemes fail to retain maximal order of accuracy near the high-order
critical points \cite{bor}.  Acker et al. \cite{ABC} added a new term
in the smoothness indicator to the fifth order WENO-Z weight
to increase the relevance of less-smooth substencil such that
it achieved better resolution in the smooth part
of the solution while maintaining the same numerical stability
as the original WENO-Z at shocks and discontinuities.
Some other fifth-order WENO schemes were further proposed by
modifying nonlinear weights \cite{GMR, ha1, KHY, Zhu1}.
Sixth or higher order WENO techniques have been developed
in the literature \cite{bal, gerolymos, HKYY, HU1, HU2}.
The central WENO \cite{Cravero, Kaser, Levy99},
hybrid compact WENO schemes \cite{Pirozzoli, Shen},
and other versions of the WENO methods
\cite{bal16, Chen, Levy99, Levy, Liu_wls, Zhang, Zhu2}
have been constructed to improve the performance of the
WENO techniques.



The space of algebraic polynomials is the
most well-established tool to reconstruct numerical flux.
However, the interpolation method
cannot be regulated according to the trait of the given data such that
it causes  excessive numerical dissipation
when approximating rapidly varying  data (e.g.,  sharp gradients
or high oscillations).
To circumvent this limitation, this study exploits the interpolation method
based on the space of exponential polynomials of the form
\baa
\phi(x) =  x^k e^{\lambda x},
\quad k\in\ZZ_+,\ \lambda\in\RR \cup i\RR,
\eaa
that allows  an environment to fit the approximation to the characteristic
of the given  problem.
For a given exponential
polynomial space, the choice of the tension (or shape) parameter $\lambda$
has a significant impact on the accuracy of interpolation.
A well-selected parameter
can yield better results compared to the polynomial-based method
for various types of PDEs \cite{HLY, HKYY, YuS, ZQ1, ZQ2}.
However, selecting an optimal parameter is an
important and difficult problem, indeed a topic of active research.
Most studies end up finding the tension parameter
by using trial and error or minimization problem. 
In this regards, the goal of this study is first to present a specific type
of exponential approximation space
for the construction of numerical fluxes under the setting
of the fifth-order WENO scheme. We then introduce a practical approach
to determine an optimal parameter by taking into account
the relation between the value of the tension parameter and the accuracy of
the exponential polynomial interpolation.
As a result, the proposed WENO scheme (termed as WENO-H)
provides an improved  order of accuracy better than the other fifth-order
WENO methods.
In fact, we will observe that the sixth-order accuracy can be achieved
by the WENO-H technique, without loss of accuracy at critical points.
A rigorous analysis is provided to prove the improved convergence rate.
Further, a modified smoothness indicator based on $L^1$-norm approach
is presented along with a new global smoothness indicator.
Accordingly, the proposed WENO scheme reduces  numerical dissipation
significantly, while attaining better resolution in smooth regions.
Some experimental results for various benchmark test problems
are given to illustrate the performance of the WENO-H scheme.
The results are compared with those of
some other methods to confirm the reliability of the proposed method.

The organization of the paper is as follows. Section 2
gives a brief review of the fifth-order WENO schemes for
one-dimensional scalar conservation laws.
In section 3, we propose a specific type
of exponential approximation space
and a practical approach to determine the parameter
under the setting of the fifth-order WENO scheme.
We also give a detailed analysis for the improved order of accuracy
under a suitable condition of the tension parameter.
In section 4, we introduce new modified smoothness indicators
along with the associated WENO scheme.
Finally, section 5 presents  some experimental results
to demonstrate the performance of the WENO-H.
A conclusion is given in section 6.

\section{\bf WENO schemes}\label{SEC-WENO}

In this section we describe a general formulation of finite difference
WENO schemes for solving hyperbolic conservation laws.
Without loss of generality,
we shall focus on the one-dimensional hyperbolic conservation
laws which is given in the form
\begin{align}\label{H-EQ-1}
\begin{split}
 & q_t + f(q)_x  = 0,\quad   t\geq 0, \ x\in \RR,\\
 &     q(x,0)    = q_0(x),
\end{split}
\end{align}
with suitable boundary conditions.
Here, $q = (q_1, \cdots, q_m)$ is a vector of conserved quantities,
$f(q)$ is a vector-valued function with $m$ components,
and $x$ and $t$ indicate space and time variables respectively.

For simplicity of our presentation, we introduce some notation.
The computational domain is assumed to be uniformly distributed with
the cells $I_j=[x_{j-1/2}, x_{j+1/2}] $ and their centers $x_j$.
The points $\{x_{j+1/2}\}$ are called the cell boundaries
and the cell size is denoted by
$\Delta x = x_{j+1/2} - x_{j-1/2}$.
In particular, we use the notation $f_j$ for the function value
at the node $x_j$, i.e., $f_j:=f(x_j)$.
The set of nonnegative integers is denoted by $\ZZ_+$, i.e.,
$\ZZ_+ =\{0\}\cup\NN$.

\subsection{Formulation of WENO scheme}
At each node $x_j$, the semi-discretized form  of the equation
in \eqref{H-EQ-1} generates a system of ODE
(ordinary differential equation) by the method of lines:
\begin{equation}\label{SEMI-DIS}
      \frac{d q_j}{d t} = - \frac{\partial f}{\partial x} \Big |_{x =x_j}
\end{equation}
with $q_j(t)$ an approximate value to
the value $q(x_j,t)$ in a grid.
Defining
the flux function  $h$ implicitly by
\begin{eqnarray} \label{CELL-AVG}
    f(q(x))= \frac{1}{\Delta x}
       \int^{x+\Delta x/2}_{x-\Delta x/2} h(s)ds,
\end{eqnarray}
a conservative finite difference formulation constructs a numerical flux
$\hat f$ which approximates the function $h$ at the cell boundaries
with a high order of accuracy.
Therefore, the spatial derivative $\frac{\partial f}{\partial x}|_{x=x_j}$
in  \eqref{SEMI-DIS}
can be represented as a discrete difference of the function $h$ at
the cell boundary $x_{j+1/2}$, which also can be
exactly approximated by the following conservative scheme
\begin{align}\label{H-EQ-A}
       \frac{\partial f}{\partial x} \Big |_{x =x_j}
        & = \frac{h_{j+1/2} - h_{j-1/2}}{\Delta x}.
\end{align}
The interface numerical flux $\hat f$ can be computed by
\begin{equation} \label{num-flux}
\hat {f}_{j+1/2} = Q(f^+_{j+1/2}, f^-_{j+1/2})
\end{equation}
where $Q$ indicates a flux method.
In practice, in order to ensure the numerical stability
and avoid entropy violating solutions,
the flux $f$ is split into two components $f^+$ and
$f^-$ as $f(q) = f^+(q) + f^-(q)$.
The interface limits
$ f^-$ and $f^+ $ are
obtained by negative and positive parts of the flux $f(q)$, respectively.
This study employs the Lax-Friedrichs splitting
defined  by
\begin{equation}\label{gm10}
        f^{\pm}(q) = \frac{1}{2} \big{(}f(q) \pm \alpha q \big{)},
\end{equation}
where
$f^+$ and $f^-$ indicate the approximations to $f$ from right and left
respectively and $\alpha = \max_q |f'(q)|$ on the pertinent range of $q$.

\subsection{Fifth-order WENO schemes}


In the fifth-order WENO finite difference scheme,
the numerical flux $\hat{f}$ at the cell boundary $\xj$ in \eqref{num-flux}
is constructed on a $5$-point stencil
$$\mS_5:= \mS(j):=\{x_{j-2}, x_{j-1}, x_j, x_{j+1}, x_{j+2} \}$$
which is subdivided into three candidate substencils
$\mS_k:=\{x_{j+k-2}, \ldots, x_{j+k}\}$, $k= 0,1,2$.
Letting $\hat{f}^{k}_{j + 1/2}$
be the local solution constructed on each substencil $S_k$,
the final WENO approximation is defined by a convex combination of these
functions with weights $\omega_k$:
\begin{equation*}
\hat{f}_{j+ 1/2} = \sum_{k=0}^2 \omega_k \hat{f}^{k}_{j+ 1/2}.
\end{equation*}
To construct the weights $\omega_k$,
we first find the constants $d_k$ which are called optimal (or ideal)
weights such that its linear
combination of $\hat{f}^{k}_{j+ 1/2}$ results in the central upwind
fifth-order scheme  to $h_{j+ 1/2}$.
The specific values of $d_k$ are known as $d_0=0.1$, $d_1=0.6$ and $d_2=0.3$
\cite{shu01}.
Then the  nonlinear weights $\omega_k$ are defined by
using these numbers $d_k$ as follows:
\begin{align*}
\omega_k = \frac{\alpha_k}{\sum^{2}_{\ell=0} \alpha_\ell}, \quad
\alpha_k = \frac{d_k}{(\varepsilon + \beta_k)^2},~
\end{align*}
where a small positive value $\varepsilon>0$ is employed to
prevent the division by zero
and $\omega_0+\omega_1+\omega_2 =1$.
The local smoothness indicator $\beta_k$  estimates the
regularity of the numerical flux $\hat f^k$ which indeed determines
to what extent the solution $\hf^k$ contributes
to the final WENO reconstruction.
The smoothness indicators introduced by Jiang and Shu \cite{Jiang1} are given by
\begin{equation}\label{JS-IND}
\beta_k = \sum^{2}_{\ell=1} \int^{x_{j+1/2}}_{x_{j-1/2}} \Delta
x^{2\ell-1} \Big{(}\frac{d^\ell}{dx^\ell}\hat{f}^{k}\Big{)}^2dx.
\end{equation}
The  scheme is called WENO-JS.
It was noted that the WENO-JS achieves only  the third order accuracy
at critical points. To correct this drawback,
two different WENO techniques have been developed.
Henrick et al. \cite{hen} suggested a modified fifth-order WENO method
(WENO-M)  by using a mapping procedure to the smoothness indicators
to recover the maximal convergence rate  (\ref{JS-IND}).
Later, Borges et al. \cite{bor} introduced
another approach for the WENO scheme (referred to as WENO-Z) by
using a global high order smoothness indicator
which makes the nonlinear weights converge to the optimal weights
faster than the classical WENO scheme.

\section{\bf Interpolation based on exponential polynomial basis functions}



\subsection{\bf Exponential Function Space}

Although the space of polynomials is most commonly used to implement numerical
fluxes, the interpolation method causes excessive numerical dissipation
when approximating rapidly varying data. 
In order to make up for this weakness, we employ a method based on
exponential polynomials of the form
\ba \label{E-PHI}
\phi(x) =  x^k e^{\lambda x}, \quad k\in\ZZ_+,
\quad \lambda\in\RR \cup i\RR.
\ea
If $k=0$ and $\lambda$ is pure imaginary, the function $\phi$ becomes
a trigonometric polynomial.
The motivation of using this type of functions is to exploit $\lambda$
as a tension  parameter so that
it allows one to choose an optimized  parameter
to fit the specific features of the solution.

Let $\BB_m:=\{\phi_1,\ldots,\phi_m\}$ with $m\in\NN$ be a set of
exponential polynomials. When the set $\BB_m$  constitutes
an {\em extended Tchebysheff system} on $\RR$,
the non-singularity of the interpolation matrix is guaranteed \cite{KS}.
Practically, for a given cell boundary $\xj$,
we look for the approximate solution from the shifted function space
\ba\label{SP-E}
\Gamma_m :={\rm span} \{\phi_n(\cdot -\xj):\phi_n\in \BB_m\}
\ea
to avoid using large numbers in the interpolation matrix.
The construction of the numerical flux $\hf$ based on $\BB_5$
complies with the methodology of the central-upwind schemes.
We use an $m$-point stencil to construct $\hat f$
approximating the flux $h$ with the $m$th convergence order
at the cell interface.
That is,  from a given set of cell-average values on the stencil,
the function $\hf$ is defined as follows
\ba\label{EP-Central}
h \approx \hat f =\sum_{n=1}^m a_n \phi_n(\cdot-\xj) \in \Gamma_m
\ea
with the coefficients $a_k$ obtained by evaluating the integral at
the stencil nodes \cite{HKYY, hen}.
Equivalently, a convenient way to construct the numerical flux $\hf$ is via
Lagrange's interpolation formula
to the primitive function $H$ of $h$ on the cell-boundaries (say
$\mS_m^b:=\{x_{j-r-1/2}\ldots,x_{j-r+m-1/2}\}$ for some $r\in\ZZ_+$),
that is,
\baa
\hf_{j+1/2}
:=\sum_{n=0}^m L_n^\prime(\xj ) H(x_{j-r+n-1/2}).
\eaa
In actual computation, the function $H$ need not to be computed explicitly.
The values of $H$ at the cell boundaries can be computed directly
by using the given cell-average values.
Letting $\BB^b_m=\{\vp_0,\ldots, \vp_m\}$
be a set of exponential polynomials such that
$\Gamma_m= {\rm span} \{\vp_k'(\cdot-\xj): \vp_k\in \BB^b_m \}$,
the Lagrange functions $L_n$ are in fact determined by solving
the linear system
\ba\label{L-Rep}
\sum_{n=0}^{m}
L_n(x) \vp_\ell(x_{j-r+\ell -1/2}-\xj)=\vp_\ell(x),
\quad \forall \vp_\ell\in \BB^b_m,
\ea
which means the exponential polynomial reproducing property of $\{L_n\}$.
It is obvious that each $L'_n$ belongs to the space  $\Gamma_m$
for $n=0,\ldots, m$.

\begin{remark}\label{REM-LAG}
{\rm
When $\vp_n(x) = x^n$ for $n=0,\ldots, m$,
the solution of the linear system in \eqref{L-Rep} is
uniquely determined by the set of the Lagrange polynomials
(denoted by $\{u_n:n=0,\ldots, m\}$)
of degree $m$ on $S^b_m$ which
fulfills the polynomial reproducing property
\ba\label{L-POL}
\sum_{n=0}^m u^{(\alpha)}_n(x) p(x_{j-r+n-1/2}) = p^{(\alpha)}(x),
\quad \forall p\in\Pi_m.
\ea
For later use,  we introduce the dilation of $u_n$ that is,
$\bau_n: = u_n(\Delta x \cdot)$, which are
the Lagrange polynomials on the stencil $\{-r-\frac 12,\ldots, -r+m-\frac 12\}$.
It is necessary to remark that the Lagrange polynomials are shift-invariant
so that $\bau_n(1/2) = u_n(\xj)$.
}
\end{remark}

The relation between $L_n$ and $u_n$ is treated in the following
Lemma \ref{AS-EQ}, which is useful for our further analysis.
In fact, the specific proof can be obtained similarly as in the proof of
\cite[Theorem 3]{HLY}.
But, in order to make this paper self-contained, the proof is
sketched briefly here.

\begin{lemma}\label{AS-EQ}
Let ${\bL}(x) =(L_n(x):n=0,\ldots,m)$ and ${\bU}(x) =(u_n(x):n=0,\ldots,m)$
be the vectors of Lagrange functions in \eqref{L-Rep} and \eqref{L-POL}
respectively. Then, for any $\alpha=0,\ldots, m-1$, we have
$$\|{\bL}^{(\alpha)}(\xj)  - {\bU}^{(\alpha)}(\xj) \|_\infty = \mO(\Delta x).$$
\end{lemma}
\begin{proof}
For notational simplicity, put $\bax =\xj$.
Let $T_{\vp_k}$ be the Taylor polynomial of $\vp_k(\cdot -\bax)$
up to degree $m$ around $\bax$, i.e.,
$$T_{\vp_k}=\sum_{\ell =0}^m(\cdot-\bax)^{\ell}\phi^{(\ell )}_k(0)/\ell !,$$
and let $\bT$ be the matrix with components
$\bT(k,n)= T_{\vp_k}(x_{j-r+n-1/2})$
for $k,n =0,\ldots, m$.
Further, letting $\bD$ be the diagonal matrix with the entries
$\bD={\rm diag}( \Delta x^k: k=0,\ldots, m)$,
the matrix $\bT$   can be written as
$$ \bT = \bW\cdot\bD\cdot \bV$$
with
$${\bf W} = (\vp_k^{(\ell)}(0): k, \ell=0,\ldots, m), \quad
\bV= ((-r+n-1)^\ell/\ell!:\ell, n=0,\dots m).$$
Here $\bV$ is a Vandermonde matrix and
$\bW$ is the Wronskian matrix of $\BB_5$ so that their non-singularities
are clear.
Using this expression,
the linear system in (\ref{L-Rep})
which in fact uniquely determines
the solution ${\bf L}(\bax)$
can be decomposed into the form
$$ (\bV+ \Delta x \bR) {\bf L}^{(\alpha)}(\bax)
  = \bp^{(\alpha)}(\bax) + \Delta x {\br}^{(\alpha)}(\bax)$$
with
${\bp}^{(\alpha)}(\bax) =(\delta_{\alpha,n}: n=0,\ldots, m)^T$
for some matrices $\bR$ and $\br(x)$ with
$\|\bR \|_\infty, \|\br\|_\infty \leq c_1<\infty$.
It is well-known (e.g., see \cite{GL})
that a $\mO(\Delta x)$ perturbation of
a non-singular matrix results in also the $\mO(\Delta x)$ perturbation of
its inverse matrix.
Thus, it follows that
\begin{align*}
\begin{split}
 {\bL}^{(\alpha)}(\bax)
&= (\bV+ \Delta x \bR)^{-1}(\bp^{(\alpha)}(\bax)
       +\Delta x {\br}^{(\alpha)}(\bax)) \\
&= (\bV^{-1}+ \Delta x \widetilde {\bR})
   (\bp^{(\alpha)}(\bax)+\Delta x {\br}^{(\alpha)}(\bax)) \\
&= \bV^{-1} \bp^{(\alpha)}(\bax) + \mO(\Delta x).
\end{split}
\end{align*}
In view of \eqref{L-POL},
$\bV^{-1} \bp^{(\alpha)}(\bax)={\bU}^{(\alpha)}(\bax)$. It leads to
$\|{\bL}^{(\alpha)}(\bax)  - {\bU}^{(\alpha)}(\bax) \|_\infty = \mO(\Delta x)$,
which completes the proof.
\end{proof}

\subsection{\bf Optimal tension parameter}

The goal of this section is two folds. We first propose a specific type
of exponential approximation space
for the construction of numerical fluxes under the setting
of the fifth-order WENO scheme.
This study is particularly interested in the following set of functions
\begin{align}\label{E-SP}
\begin{split}
& \BB_5:=\{1,x,x^2, \phi_3(x), \phi_4(x)\},
\end{split}
\end{align}
where $\phi_3$ and $\phi_4$ are exponential polynomials.
In this study, we will mainly concentrate on the case
$$\phi_3(x) = \sinh \lambda x,\quad \phi_4(x)= \cosh \lambda x .$$
As discussed before, for a given cell-boundary $\xj$,
the approximate solution on $\mS_5$ is obtained from the space
$$\Gamma_5={\rm span} \{\phi_i(\cdot -\xj) :\phi_i\in\BB_5\}$$
to avoid using large numbers in interpolation process.
Then our next goal is to present a practical approach to find the
parameter $\lambda$ without any trial and error or minimization process.
For this purpose, we take into account the relation between
the parameter $\lambda$ and
the convergence behavior of the approximation
to the spartial derivative $\partial f/\partial x$ at $x=x_j$
in \eqref{H-EQ-A}, i.e.,
\begin{align}\label{ERR-0}
\CE_j :=
\frac{\partial f}{\partial x}\big |_{x=x_j}
- \frac{\hat{f}_{j+1/2} - \hat{f}_{j-1/2}}{\Delta x}.
\end{align}

\medskip \noindent
Our specific selection of  $\phi_n$ ($n=3,4$) and the associated
tension parameter is presented below
in terms of the primitive function $H$ of the flux $h$.
In actual computation, the values of $H$ at the cell boundaries can be
computed directly by using the given cell-average values:
\ba\label{HXJ}
H(x_{n+1/2}) =\Delta x\sum_{\ell=j-2}^n \bar h_\ell .
\ea
We then  verify that for a suitably chosen parameter,
the corresponding  interpolation method can improve the rate of accuracy of
the classical polynomial interpolation method.

\begin{remark}\label{REM-H6}
{\rm
Prior to further study, it is worthwhile to point out that
if $|H^{(6)}(x)| = 0$
(or practically, $|H^{(6)}(x)| \leq {\bar c} \Delta x^2$
for a fixed constant $c>0$),
the interpolation method provides an improved accuracy of $\CE_j$
for any suitable set $\BB_5$ (including algebraic polynomials);
see Proposition \ref{COR-C1}.  In this case,
one may use the classical interpolation method based on polynomials
to construct $\hf_{j+1/2}$.
In this view point, in what follows, it is reasonable to consider the case
$H^{(6)}(x)\not= 0$.
}
\end{remark}

\medskip \noindent
$\bullet$ {\bf Central Condition A.}
For a given cell boundary $\xj$,
without great loss, we suppose that $H^{(n)}(\xj)$ does not vanish
simultaneously for both $n=4,5$. Then, exponential approximation
space is chosen by considering the following two cases:

\begin{itemize}
\item[{\bf C1:}] If $H^{(4)}(x_{j+1/2})$ is non-zero,
we set
$$ \BB_5=\{1, x, x^2, \sinh \lambda x, \cosh \lambda x\}$$
with the tension parameter  $\lambda$ satisfying the condition
$$\lambda^2 =\Big ( \frac{H^{(6)}}{H^{(4)}}\Big )(x_{j+1/2}) +O(\Delta x).$$
\noindent
In practice, as long as the flux $f$ is not constant or linear
(more generally, polynomially changing) around $\xj$,
$H^{(4)}$ is nonzero almost everywhere.
Hence, in this study, we are mainly concentrating on the case C1.
But, if this is not the case, it is treated by the case C2.

\item [{\bf C2:}]
If $H^{(4)}(x_{j+1/2})=0$ and  $H^{(5)}(\xj)$ is nonzero,
we set
$$\BB_5=\Big \{1, x, x^2, \sinh \lambda x,
\cosh \lambda x+\frac{\lambda^6x^5}{5!}\Big \}$$
with the tension parameter $\lambda$ satisfying the condition
$$\lambda^2 =\Big ( \frac{H^{(6)}}{H^{(5)}}\Big )(x_{j+1/2}) +O(\Delta x).$$
\end{itemize}

\begin{remark}\label{REM-H6-2}
{\rm
For the construction of local numerical flux on each substencil $\mS_k$ for
$k=0,1,2$, we use the algebraic polynomials, i.e.,
$$\BB_3=\{1,x,x^2\}.$$
It means that the reconstruction of the local solution on
each substencil $S_k$ is exactly the same as the case of
the classical fifth-order WENO method.
}
\end{remark}

\subsection{\bf Improved approximation order by exponential polynomials}

We now prove that the proposed interpolation method based
on the `Central Condition A' provides an improved accuracy compared
to other fifth-order WENO schemes.
To do this,
let $\BB^b_5=\{\vp_0,\ldots, \vp_5\}$
be a set of exponential polynomials such that
$\Gamma_5= {\rm span} \{\vp_k'(\cdot-\xj):k=0,\ldots, 5\}$.
Then the numerical flux $\hf$ is defined through
the Lagrangian interpolation formula
to the function $H$ on the cell-boundaries
$\mS_5^b:=\{x_{j-5/2},\ldots,x_{j+5/2}\}$, that is,
\begin{align}\label{Lf-Lag}
\hf_{j+1/2}
:=\sum_{{\color{red} n=0}}^{5} L^\prime_{n}(\xj ) H(x_{n+j-5/2}).
\end{align}
First consider the case `C1'. The case `C2' follows later.



\def\CE{\mathcal E}



\medskip \noindent
$\bullet$ {\bf Case I:} $H^{(4)}(x_{j+1/2})$ is non-zero.

\medskip \noindent
Recalling that $ \BB_5=\{1, x, x^2, \sinh \lambda x, \cosh \lambda x\}$,
let $\BB^b_5=\{\vp_0,\ldots, \vp_5\}$
be a set of exponential polynomials such that
$\Gamma_5= {\rm span} \{\vp_n'(\cdot -\xj):\phi_n\in \BB_5^b\}$.
To facilitate our further analysis for the convergence order
of the proposed method,
we reorganize the elements in $\BB_5^b$ as follows:
\begin{align*}
&\vp_n(x)={x^n}/{n!} \  (n=0,\ldots, 3), \\
&\vp_4(x) =\frac{1}{\lambda^4}
     \Big (\cosh (\lambda x) -1 -\frac{(\lambda x)^2}{2}\Big ),\quad
\vp_5(x)= \frac{1}{\lambda^5}
     \Big (\sinh (\lambda x)  -\lambda x -\frac{(\lambda x)^3}{3!}\Big ).
\end{align*}
It is obvious that  each function
$\vp_n'(\cdot-\xj)$ belongs to the space  $\Gamma_5$.
Then, by a linear combination of these functions,  we define
an auxiliary function  $\psi $ as follows:
\ba\label{PSI}
\psi:= \psi_j := \sum_{n=0}^{5}\mu_{j,n}\vp_n (\cdot-\xj)
\ea
with the coefficient vector ${\mathbold \mu}= (\mu_{j,n}: n =0,\ldots, 5)^T$
obtained by solving the linear system
\ba\label{PSI-2}
\psi^{(n)}(\xj)=H^{(n)} (\xj), \quad n=0,\cdots, 5.
\ea
The following lemma treats the uniqueness of the solution ${\mathbold \mu}$
and also finds its  explicit form.

\begin{lemma}\label{LEM-C1}
Let $\psi$ be defined as  in \eqref{PSI} with
the coefficient vector ${\mathbold \mu}= (\mu_{j,n}: n =0,\ldots, 5)^T$.
Then, there exists a unique solution ${\mathbold \mu}$ with the form
$\mu_{j,n} = H^{(n)}(x_{j+1/2})$ for $n=0,\ldots,5$.
\end{lemma}
\begin{proof}
Let ${\bf W}_0:=(\vp^{(\ell)}_{n}(0): \ell,n = 0,\dots,5)$
be the Wronskian matrix of $\{\vp_0,\ldots, \vp_5\}$ at $0$
and let
${\bf H}_j:=  (H^{(\ell)}(\xj): n = 0,\dots,5) $.
Note that the vector ${\mathbold \mu }$
can be rewritten in the following matrix form
$$ {\bf W}_0\cdot{\mathbold \mu }= {\bf H}_j.$$
Since ${\bf W}_0$ is non-singular, the uniqueness of the solution
${\mathbold \mu}$ is obvious. In fact, an elementary calculation reveals that
$\vp^{(\ell)}_n=\delta_{\ell,n}$ with $\delta_{\ell,n}$ the Kronecker delta,
which means that ${\bf W}_0$ the identity matrix.  Thus,
the lemma is proved immediately.
\end{proof}

We now prove the convergence order of the approximation
to the spatial derivative ${\partial f}/{\partial x}$ in \eqref{ERR-0}.
This study is especially interested in approximating functions $g$
in the Sobolev space
\begin{eqnarray*}
	W_\infty^k(\Omega):=\big \{g\in C^{(k)}(\Omega) :
	\|g^{(k)}\|_{L^\infty(\Omega)} < \infty \}
\end{eqnarray*}
where  $\Omega$ is an open set in $\RR$.
For this proof, we recall that $g_j$ indicates the value $g(x_j)$
at the node $x_j$.
Also, denote by $T_{g}$ the Taylor polynomial of degree $5$
around $x_{j+1/2}$ of the function $g$, i.e.,
\ba\label{T-P-10}
T_g:=T_{g,\xj}: =\sum_{n=0}^5 (\cdot -\xj)^n g^{(n)}(\xj).
\ea

\begin{theorem} \label{TH-A-1}
Assume that $H\in W_\infty^{7}(\Omega)$ with $\Omega$ an open
neighborhood  of $\xj$.
Let $\hf$ be the numerical flux  defined as in \eqref{Lf-Lag}.
Then,  under the `Central Condition A-C1', we have
\ba\label{TH-EQ-0}
h(x_{j\pm 1/2})- \hat{f}(x_{j\pm 1/2})
= C_j \Delta x^{6} + \mO(\Delta x^7)
\ea
with
\ba\label{CJ-0}
C_j= \Big (\frac{H^{(5)} H^{(6)}}{H^{(4)}} - H^{(7)}\Big )(x_{j})
     \sum_{n=0}^5 \bau'_n(1/2){(n-3)^7}/7!.
\ea
where $\bau_n$ are the Lagrange polynomials of degree $5$ on
the stencil $\{-5/2,\ldots, 5/2\}$.
\end{theorem}
\begin{proof}
In this proof, we first analyze the accuracy of $\hat f$ to the function $h$
at $\xj$. To do this, we employ the auxiliary function $\psi$
defined in \eqref{PSI}.  Due to the condition in  \eqref{PSI-2},
$\psi'(x_{j+1/2}) = H'(x_{j+1/2})$. Also, since $H$ is the primitive function
of $h$, $H'(x_{j+1/2}) = h_{j+1/2}$. It implies that
$h_{j+1/2}= \psi'(x_{j+1/2})$. Then using the formula of the numerical flux
$\hat f_{j+/2}$ in \eqref{Lf-Lag}, we can write
\begin{align} \label{ERR-30}
\begin{split}
h_{j+1/2} - \hat{f}_{j+1/2}
& =\psi'(x_{j+1/2}) - \sum_{n=0}^5 L'_n(x_{j+1/2}) H(x_{n+j-5/2}).
\end{split}
\end{align}
Further, since the derivative $\psi'$ belongs to the space $\Gamma_5$,
in view of the exponential polynomial
reproducing property in \eqref{L-Rep}, we can express
\baa
\psi'(x_{j+1/2})
= \sum_{n=0}^5 L'_n(x_{j+1/2})\psi(x_{n+j-5/2}).
\eaa
Combining this with \eqref{ERR-30} derives the equation
\begin{align}\label{ERR-51}
h_{j+1/2} - \hat{f}_{j+1/2}
&= \sum_{n=0 }^5 L'_n(x_{j+1/2})(\psi(x_{n+j-5/2})- H(x_{n+j-5/2})).
\end{align}
Next, to estimate the difference $\psi (x_{n+j-5/2}) - H(x_{n+j-5/2})$
in the above equation, we use the Taylor expansion argument.
In fact, since $\psi^{(\ell)}(x_{j+1/2})=H^{(\ell)}(x_{j+1/2})$
for $\ell=0,\dots, 5$, it is apparent that
$$T_{\psi}=T_H$$
with $T_{g}$ the Taylor polynomial of $g$ in \eqref{T-P-10}.
Accordingly, it holds that
\ba\label{R-EQ}
\psi(x_{n+j-5/2})- H(x_{n+j-5/2})
= R_\psi(x_{n+j-5/2})-R_{H}(x_{n+j-5/2})
\ea
where $R_{g}$ is the remainder of the Taylor polynomial $T_g$.
Then, in order to get an improved convergence rate of
the difference $h_{j+1/2}-\hf_{j+1/2}$,  we would like to verify that
$$\psi^{(6)}_j(x_{j+1/2})=H^{(6)}(x_{j+1/2})$$
under the `Central Condition A-C1'.
Indeed, from the formula of $\psi$ in \eqref{PSI-2} and Lemma \ref{LEM-C1},
a direct calculation yields the identity
$\psi^{(6)}_j(x_{j+1/2}) =\lambda^2  H^{(4)}(x_{j+1/2})$,
where $H^{(4)}(x_{j+1/2})$ is non-zero by assumption.
Thus, putting
\ba\label{LAM-20}
\lambda^2 =  H^{(6)}(x_{j+1/2})/ H^{(4)}(x_{j+1/2}),
\ea
we prove that
$\psi^{(6)}(x_{j+1/2}) = H^{(6)} (x_{j+1/2}).$
Consequently, using the explicit formula of the remainder terms of $\psi$ and
$H$, it holds immediately  from \eqref{R-EQ} that
\begin{align}\label{ERR-52}
\begin{split}
\psi(x_{n+j-5/2})- H(x_{n+j-5/2})
&=\Delta x^7 \frac{(n-3)^7}{7!}
(\psi^{(7)}-H^{(7)})(x_{j+1/2}) + O(\Delta x^8).
\end{split}
\end{align}
Moreover,  from the definition of $\psi$, we calculate that
$\psi^{(7)}(\xj) =\lambda^2 H^{(5)}(\xj)$.
Substituting the value $\lambda$ in \eqref{LAM-20} into this equation
results in the expression
\begin{align*}
\begin{split}
\psi^{(7)}(\xj) =\frac{H^{(5)} H^{(6)}}{H^{(4)}}(\xj).
\end{split}
\end{align*}
Applying the mean-value theorem, it follows that
\begin{align}\label{ERR-53}
(\psi^{(7)}- H^{(7)})(\xj)
=\Big ( \frac{H^{(5)} H^{(6)}}{H^{(4)}}- H^{(7)}\Big )(x_j) + O(\Delta x),
\quad {\rm as}\quad \Delta x \to 0.
\end{align}
On the other hand, let us recall from Lemma \ref{AS-EQ}
that $L'_n (\xj)= u'_n(\xj) +\mO(\Delta x)$ with $u_n$
the Lagrange polynomial of degree $5$ on the stencil $\mS_5^b$
as in \eqref{L-POL}.
Also, $\sum_{n=0}^5 |u_n'(\xj)| \leq c\Delta x^{-1}$.
Combining these arguments with \eqref{ERR-51},
\eqref{ERR-52} and \eqref{ERR-53},
we arrive at the expression
\begin{align*}
h_{j+1/2} - \hat{f}_{j+1/2}
&=\Delta x^7  \sum_{n=0 }^5 u'_n(x_{j+1/2})\frac{(n-3)^7}{7!}
  \Big (\frac{H^{(5)}H^{(6)}}{H^{(4)}} - H^{(7)}\Big )(x_j) + O(\Delta x^7).
\end{align*}
Now, let $\bau_n: = u_n(\Delta x \cdot)$ be the dilation of $u_n$, that is,
the Lagrange polynomials  on the stencil $\{-\frac 52,\ldots, \frac 52\}$
as discussed in Remark \ref{REM-LAG}.
Clearly, $\bau'_n =\Delta x u'_n(\Delta x\cdot ) $ such that
$u'(\xj) = \Delta x^{-1} \bau'_n(1/2)$.
Therefore, we conclude that
\begin{align} \label{ERR-71}
\begin{split}
h_{j+1/2} - \hat{f}_{j+1/2}
& = C_j \Delta x^6  + O(\Delta x^7)
\end{split}
\end{align}
with $C_j$ defined in \eqref{CJ-0},
which is the required result of this theorem.
Moreover,
to estimate $\hat f_{j-1/2}$,
the stencil $\mS_5$ used to compute $\hf_{j+1/2}$ is moved by one-grid
to the left.
Since the Lagrange polynomials are shift-invariant,
we can prove \eqref{TH-EQ-0} by applying the same technique.
The proof is completed.
\end{proof}

\begin{corollary} \label{CO-A-2}
Assume that $H\in W_\infty^{7}(\Omega)$ with $\Omega$ an open
neighborhood  of $\xj$.
Let $\hf$ be the numerical flux  defined as in \eqref{Lf-Lag}.
Then,  under the Central Condition A-C1, we have
\begin{equation*}
\CE_j(f) =
\frac{h_{j+1/2} - h_{j-1/2}}{\Delta x}
- \frac{\hat{f}_{j+1/2} - \hat{f}_{j-1/2}}{\Delta x} = O(\Delta x^{6})
\end{equation*}
\end{corollary}
\begin{proof}
The term $C_j \Delta x^6$ in \eqref{TH-EQ-0} is the same
for both $h_{j+1/2} -\hf_{j+ 1/2}$
and $h_{j- 1/2} -\hf_{j- 1/2}$.
Thus,
putting the result of Lemma \eqref{TH-A-1} at the finite difference formula
$\CE_j(f)$ in \eqref{TH-A-1}, we find that  $\Delta x^6$ term remains
after division by $\Delta x$.
Thus, the theorem holds immediately.
\end{proof}

\noindent
$\bullet$ {\bf Case II:} $H^{(4)}(x_{j+1/2})=0$ and
$H^{(5)}(x_{j+1/2})$ is non-zero.

\medskip \noindent
The general approach for this case  is similar
to the Case I, but we have to modify it to meet the condition
$H^{(4)}(x_{j+1/2})=0$ and $H^{(5)}(x_{j+1/2})\not = 0$.
For this purpose,
as before,  we employ an auxiliary function  $\psi $ defined
by a linear combination of the functions in $\BB_5^b$.
As in the case of C1, we reorganize  the elements in
$\BB_5^b$ as follows:
\begin{align*}
&\varphi_n(x) = x^n/ n!,\quad   (n=0,\ldots, 3), \\
&\vp_4(x)=\frac1{\lambda^4}
     \Big (\cosh (\lambda x)-1-\frac{(\lambda x)^2}{2} \Big ), \quad
\vp_5(x) = \frac1{\lambda^5}
\Big (\sinh(\lambda x) - \lambda x - \frac{(\lambda x)^3}{3!} \Big )
  +\frac{\lambda^2 x ^6}{6!}.
\end{align*}
It is not difficult to see that $\vp_n'(\cdot-\xj)\in\Gamma_5$.
Compared to the Case I,
we note that only the function $\vp_5$ is defined differently.
We then introduce an auxiliary $\psi $ by
\ba\label{PSI-20}
\psi := \psi_j (x) := \sum_{n=0}^{5}\mu_{j,n}\vp_n (\cdot-\xj)
\ea
with the coefficient vector ${\mathbold \mu}$
satisfying  the linear system
\ba\label{PSI-LS2}
\psi^{(\ell)}(\xj)=H^{(\ell)} (\xj),\quad \ell=0,\cdots, 5.
\ea
The uniqueness of the solution ${\mathbold \mu}$
and its  explicit form are discussed below.

\begin{lemma}\label{LEM-C2}
Let $\psi$ be defined as  in \eqref{PSI-20} with
the coefficient vector ${\mathbold \mu}= (\mu_{j,n}: n =0,\ldots, 5)^T$.
	Then, there exists a unique solution ${\mathbold \mu}$ with the form
	$\mu_{j,n} = H^{(n)}(x_{j+1/2})$ for $n=0,\ldots,5$.
\end{lemma}
\begin{proof}
Let ${\bf W}_0:=(\vp^{(\ell)}_{n}(0): \ell,n = 0,\dots,5)$
be the Wronskian matrix of $\{\vp_0,\ldots,\vp_5\}$ at $0$.
It can be easily checked that ${\bf W}_0$ is the identity matrix.
Thus, the same technique in Lemma \ref{LEM-C1} can be applied to
prove ${\mathbold \mu }={\bf H}_j$ with
${\bf H}_j:=  (H^{(\ell)}(\xj): \ell = 0,\dots,5)$.
\end{proof}

\begin{theorem} \label{TH-A}
Assume that $H\in W_\infty^{7}(\Omega)$ with $\Omega$ an open
neighborhood  of $\xj$.
Let $\hf$ be the numerical flux  defined as in \eqref{Lf-Lag}.
Then,  under the Central Condition A-C2, we have
\ba\label{TH-EQ-01}
h(x_{j\pm 1/2})- \hat{f}(x_{j\pm 1/2})
= C_j \Delta x^{6} + \mO(\Delta x^7)
\ea
with
$$C_j= \big (H^{(6)}- H^{(7)}\big )(x_{j})
     \sum_{n=0}^5 \bau'_n(1/2){(n-3)^7}/7!,
$$
where $\bau_n$ are the Lagrange polynomials  on the stencil
$\{-\frac 52,\ldots, \frac 52\}$.
\end{theorem}
\begin{proof}
The general technique for this proof is similar to that for
Theorem \ref{TH-A-1}.
Therefore, it is sketched here by pointing out
the crucial different parts.
First, since the function $\psi'$ in \eqref{PSI-20} belongs
to the space $\Gamma_5$,
as in the proof of Theorem \ref{TH-A-1}, we can write
\begin{align}\label{ERR-50}
h(x_{j+ 1/2})- \hat{f}(x_{j + 1/2})
&= \sum_{n=0}^5 L'_n(x_{j+1/2})(\psi(x_{n+j-5/2})- H(x_{n+j-5/2})).
\end{align}
Then, to estimate the term $\psi(x_{n+j-5/2})- H(x_{n+j-5/2})$
of the above equation,
we exploit the Taylor expansion argument and the condition
$\psi^{(\ell)}(x_{j+1/2})=H^{(\ell)}(x_{j+1/2})$ with $\ell=0,\dots, 5$
such that it leads to  the expression
\begin{align}\label{ERR-60}
\psi(x_{n+j-5/2})- H(x_{n+j-5/2})
=R_\psi(x_{n+j-5/2})- R_H(x_{n+j-5/2}).
\end{align}
Now, in order to obtain an improved convergence rate in \eqref{ERR-50},
we discuss the condition of the parameter $\lambda$ that makes
$\psi^{(6)}(x_{j+1/2}) = H^{(6)} (x_{j+1/2}).$
Indeed, due to Lemma \ref {LEM-C2} and the condition of $\psi$,
a direct calculation
yields the equation
$$\psi^{(6)}(x_{j+1/2}) = \lambda^2 (H^{(4)}+ H^{(5)})(x_{j+1/2})
                        = \lambda^2  H^{(5)}(x_{j+1/2})$$
because $H^{(4)}(x_{j+1/2})=0$.
By assumption, $H^{(5)}(x_{j+1/2})$ is non-zero. Hence,
putting
\ba\label{LAM-2}
\lambda^2 =  H^{(6)}(x_{j+1/2})/H^{(5)}(x_{j+1/2}),
\ea
induces the equation
$\psi^{(6)}(x_{j+1/2}) = H^{(6)} (x_{j+1/2}).$
Also, using  \eqref{LAM-2} and by the definition of $\psi$,  we obtain
$\psi^{(7)}(x_{j+1/2}) = H^{(6)} (x_{j+1/2})$.
Therefore, following the same techniques in  the proof of Theorem \ref{TH-A},
we can finish the proof.
\end{proof}

As in Corollary \ref{CO-A-2}, we get the following result.

\begin{corollary} \label{TH-A2}
Assume that $H\in W_\infty^{7}(\Omega)$ with $\Omega$ an open
neighborhood  of $\xj$.
Let $\hf$ be the numerical flux  defined as in \eqref{Lf-Lag}.
Then,  under the Central Condition A-C2, we have
\begin{equation*}
\frac{\partial f}{\partial x}\Big |_{x=x_j}
- \frac{\hat{f}_{j+1/2} - \hat{f}_{j-1/2}}{\Delta x}
= O(\Delta x^{6}).
\end{equation*}
\end{corollary}
As mentioned in Remark \ref{REM-H6},
when $H^{(6)}(\xj)=0$ or  $|H^{(6)}(x_{j\pm 1/2})|\leq {\bar c} \Delta x^2$,
the interpolation method provides an improved accuracy of $\CE_j$
for any choice of $\BB_5$ (including algebraic polynomials).
Next proposition treats this case.

\begin{proposition}\label{COR-C1}
Suppose that $|H^{(6)}(x_{j\pm 1/2})| \leq {\bar c} \Delta x^2$ for a fixed
constant ${\bar c}>0$.
Then for any choice of the set $\BB_5$ in the `Central Condition A' or
$\BB_5=\{1,\ldots, x^4\}$, we have
the estimate $|\CE_j(f) | = \mO(\Delta x^6)$ as $\Delta x \to 0$.
\end{proposition}
\begin{proof}
We first consider the case that $\hf_{j+1/2}$ is constructed by using
the classical polynomial interpolation method.
Let $T_{H}$ be the Taylor polynomial of $H$ around $\xj$ of degree $5$
and write $H =  T_{H} +  R_{H}$ with $R_{H}$
the remainder of the Taylor polynomial $T_{H}$.
Then, due to the polynomial reproducing property of the Lagrange polynomials
$\{u_n:n=0,\ldots, 5\}$ in \eqref{L-POL}, we have
\begin{align} \label{ERR-40}
\begin{split}
\hat f_{j+1/2}
&= \sum_{n=0}^5 u'_n(x_{j+1/2}) (T_H + R_H)(x_{n+j-5/2})  \\
&= T'_{H}(x_{j+1/2}) + \sum_{n=0}^5 u'_n(x_{j+1/2})R_{H}(x_{n+j-5/2}).
\end{split}
\end{align}
Obviously, $T_H'(\xj)= H'(\xj)$ and $H'(\xj) = h(\xj)$
because $H$ is the primitive function of $h$.
Also, by assumption, $|H^{(6)}(\xj)|\leq {\bar c}\Delta x^2 $ and
$H^{(7)}(\xj)= |H^{(7)}(x_j)| +\mO(\Delta x)$. It implies that
the remainder $R_H$ is the form
\ba\label{R-H}
R_{H} =(\cdot -\xj )^7 H^{(7)}(x_j)/7! +\mO(\Delta x^8).
\ea
Since  $u'_n(\xj) =\Delta x^{-1} \bau'_n(1/2)$
with $\bau_n$ the Lagrange polynomials on the stencil
$\{-\frac 52,\ldots, \frac 52\}$,
in view of these arguments with \eqref{ERR-40} and \eqref{R-H},
it holds immediately that
\begin{align} \label{ERR-35}
h_{j+1/2}- \hat f_{j+1/2} =C_j  \Delta x^6 +\mO(\Delta x^7)
\end{align}
with the constant $C_j$ defined by
\ba
C_j = -\sum_{n=0}^5 \bau'_n(1/2)\frac{(n-3)^7}{7!}H^{(7)}(x_j).
\ea
Second, suppose that $\hf_{j+1/2}$ is obtained from the space spanned by
the set $\BB_5$ either in the case `C1' or `C2'.
Since $|H^{(6)}(x_{j\pm 1/2})| \leq {\bar c} \Delta x^2$,
a direct calculation from the definition of $\psi$
and the value of $\lambda$ in the `Central Condition A'
yields the bound $|\psi^{(7)}(\xj)| \leq c|\lambda|^2 \leq c\Delta x$.
It leads to the same estimate in \eqref{ERR-35}.
Therefore,
following the same methodology in the proof of Corollary \ref{CO-A-2}, we
can get the required result
$|\CE_j(f) | =\mO(\Delta x^6)$. The proof is completed.
\end{proof}

\subsection{\bf Algorithm}\label{ALG}

The algorithm for choosing the exponential approximation space and the
tension parameter is described as follows.
Without great loss, we suppose that
$H^{(n)}(\xj)$
does not vanish simultaneously for both $n=4,5$.

{\sf
\bigskip
\noindent
\hskip -0.truein\HRule

\noindent
Algorithm for choosing the tension parameter.

\vskip -0.1truein
\noindent
\hskip -0.truein\HRule

\noindent
Let $\mS_5$ be the $5$-point stencil around the given evaluation point
$\bax= \xj$.
From the given cell-average values $\bar {h}_n$ on $\mS_5$,
construct $H (x_{n+1/2})$ on the cell boundaries
and evaluate $H^{(\ell)}(\bax)$ for $\ell=4,5,6$ by using
the $\ell$th order divided difference around $\bax$,
denoted by $[H^{(\ell)}(\bax)]$.

\begin{enumerate}
\item [{\bf 0.}]
If $[H^{(6)}(\bax)]  =0 $,
we use the classical method based on algebraic polynomials, i.e.,
$$\BB_5=\{x^n:n=0,\ldots,4\}.$$
\item [{\bf 1.}]
If $[H^{(4)}(\bax)] \not= 0$, 
we choose the set of exponential polynomials as
$ \BB_5= \{1, x, x^2, \sinh\lambda x, \cosh\lambda x\} $
with
$$ \lambda^2={[H^{(6)}(\bax)]}/{[H^{(4)}(\bax)]}.$$
\item [{\bf 3.}]
If $[H^{(4)}(\bax)]=0$ and $ |[H^{(5)}(\bax)] | >0 $,
we modify $\BB_5$ as
$ \BB_5= \{1, x, x^2, \sinh\lambda x, \cosh\lambda x + \lambda^6 x^5/5!\}$
with
$$ \lambda^2={[H^{(6)}(\bax)]}/{[H^{(5)}(\bax)]}. $$
\end{enumerate}
In practice, $H^{(4)}(\bax) \not= 0$ almost everywhere, as long as
the flux $f$ is not constant or linear
(more generally, polynomially changing) around $\bax$.
Hence, we suggest to implement the proposed algorithm mainly based on `Step 1'.
}
\noindent
\hskip -0.truein\HRule

\section{A WENO scheme improving fifth-order accuracy}

Let $\xj$ be a given cell-boundary point.
The five-point stencil $\mS_5= \{x_{j-2},\dots, x_{j+2} \}$
around $\xj$ is divided into three candidate substencils
$\mS_k$ with $k=0,1,2$ consisting of three points.
A local numerical flux $\hat{f}^{k}(x)$ is computed
in each substencil $\mS_k$ and these solutions
are combined into a weighted average
to define a final WENO approximation to the value $h_{j+ 1/2}$:
\begin{equation}\label{WA-0}
\hat{f}_{j+ 1/2}= \sum_{k=0}^2 \omega_k \hat{f}^{k}_{j+ 1/2}.
\end{equation}
In WENO reconstruction,
the nonlinear weights are required to be close to the optimal weights
for each local solution in smooth areas to attain a maximal accuracy,
while removing the contribution of stencils that contain a singular point.
From this view point, we first introduce  new optimal
weights based on the space $\Gamma_5$ of
exponential polynomials.

\subsection{An optimal weights based exponential polynomials}

For the given cell-average values on the stencil $\mS_5$,
the (global) numerical flux $\hat f_{j+1/2}$ approximating $h_{j+1/2}$
can be expressed as
\ba\label{NF-Five}
\hat f_{j+1/2}
=  \sum_{\ell=0}^{4} C_{\ell} {\bar h}_{j-2+\ell}, \quad
{\rm with}\quad
C_\ell := \Delta x  \sum_{n=\ell }^4 L^\prime_{n}(x_{j+1/2}).
\ea
The local solution  $\hat f^k_{j+1/2}$ is also computed at each substencil
$\mS_k$ with $k=0,\ldots,3$ and it
is of the form
\ba\label{NF-Three}
\hat f^k_{j+1/2}
=  \sum_{\ell =0}^{2} C_{\ell}^k {\bar h}_{j-2+k+\ell} 
\quad {\rm with}\quad
 C_{\ell}^k =\Delta x\sum_{n=\ell}^2 u_n'(\xj).
\ea
It is necessary to remark that the local numerical flux is the same as
the case of the classical fifth-order WENO scheme.
Then the numerical flux $\hat f_{j+1/2}$ can be expressed
as a convex combination of the local fluxes:
$$\hat{f}_{j+1/2} = d_0 \hat{f}^0_{j+1/2} +
d_1 \hat{f}^1_{j+1/2} +
d_2 \hat{f}^2_{j+1/2}$$
where  $\{d_k\}$ are the so-called optimal (ideal) weights such that
$d_0 + d_1 + d_2 =1$.
The optimal weights $d_k$, $k=0,1,2$, for the proposed WENO scheme
can be obtained as
\ba\label{ID-W}
d_0=C_0/C^0_0, \  d_1=(C_1 - d_0 C^0_1)/C_0^1,\  d_2=1-d_0-d_1.
\ea
Unlike the case of the classical WENO scheme,
the optimal weights $\{d_k\}$ of the proposed WENO method
may vary depending on the choice of the parameter $\lambda$
but tends to the original ideal weights as $\Delta x \to 0$.

\subsection{A New Nonlinear Weight}

The smoothness indicator is one of the most important ingredient
in WENO reconstruction because
the nonlinear weights are determined by measuring the smoothness of the
local solution on each substencil $\mS_k$.
In this section, we introduce a new set of nonlinear weights
which improves the known fifth-order  WENO schemes.
We follow the methodology of the WENO-Z scheme but provide fundamental
modifications.  A new global smoothness indicator is incorporated into
the local smoothness indicator which measures the approximate magnitude
of the derivatives of the local solution on each substencil
based on $L^1$-norm \cite{ha1}.
Specifically, let  $\ID_{n,k}$ be the operators defined by
\begin{align}
\label{N-LIM-2}
\begin{split}
& \ID_{1,k}f:=(1-k)f_{j-2+k}+(2k-3)f_{j-1+k}+(2-k)f_{j+k}, \\
& \ID_{2,k}f:= f_{j-2+ k}-2f_{j-1+k} +f_{j+k}.
\end{split}
\end{align}
Here, the operator $\ID_{1,k}f$ is a generalized undivided difference
of $f$  which approximates
$\Delta x f'$ at $x_{j+1/2}$ with higher convergence rate \cite{ha1}:
\ba\label{APP}
\ID_{1,k}f =f'(x_{j+1/2})\Delta x + \mO(\Delta x^3).
\ea
Then
the smoothness indicators $\beta_k$ are defined as follows:
\begin{eqnarray}\label{N-LIM}
\beta_k :=
\theta \left| \ID_{1,k}f\right|
+ \left| \ID_{2,k}f\right|, \quad \xi \in (0, 1],
\end{eqnarray}
where the value $\theta$ is a balanced trade off between
$\ID_{1,k}f$ and $\ID_{2,k}f$.
Having performed numerical experiments with several
alternatives, we take
$\theta=0.25$ for all test problems
except the case of  $1$-D linear advection equation in which $\theta = 0.1$.
A novel idea of the proposed nonlinear weights is to measure
the higher order information of the numerical flux
on the large stencil $\mS_5$ by using the fourth-order  undivided difference
$$\tau_5 := \ID_4 f_j 
=  f_{j-2} - 4f_{j-1} + 6f_{j} - 4f_{j+1} + f_{j+2}.$$
With these (local and global) smoothness indicators at hand,
the (unnormalized) nonlinear weight $\alpha_k$, $k=0,1,2$,  are computed
as
\begin{align}\label{ALPHA-0}
\alpha_k = {d_k} \left(1 + \frac{\tau^2_5}{\beta_k^2 + \epsilon}\right),
\quad
\epsilon: = \epsilon(\Delta x).
\end{align}
Here, $\epsilon>0$ is usually employed to prevent the denominator
from a division by zero but it in fact affects the order of accuracy of
the WENO method especially at the critical points.
The specific choice of $\epsilon $ will be discussed
in Proposition \ref{LEM-OMEGA}.
Then, the final weights $\omega_k$ are defined via the normalization
process, i.e.,
\ba\label{WEIGHT}
\omega_k = \frac{\alpha_k}{\sum^{2}_{\ell=0} \alpha_\ell},
\quad k =0,1,2.
\ea

\subsection{Convergence Order of WENO-H}

It is basic to require that the  numerical solution $\hat f_{j\pm1/2}$
approximates the flux $h$ in \eqref{CELL-AVG}
with a suitable convergence order on  smooth regions.
For this, the nonlinear weights $\omega_k$
should converge to the optimal weights $d_k$ as $\Delta x \to 0$.
To attain the sixth-order accuracy of the numerical
flux $\hat{f}_{j+1/2}$, the nonlinear weights need to satisfy
the following sufficient condition (e.g., see \cite{ha1})
\begin{align}\label{WD-COND}
\begin{split}
\omega_k^{\pm} - d_k =O(\Delta x^4),
\end{split}
\end{align}
where superscript `$\pm$' on the weight $\omega_k$ corresponds to
their use in the substencils of the local solution $f^k_{j\pm 1/2}$
respectively.
In what follows, we show that the new nonlinear weights $\omega_k$
fulfill the condition in \eqref{WD-COND}.
For this purpose,
it is helpful to introduce the general form of
$\beta_k$ which can be obtained by using the Taylor expansion argument:
\begin{align}\label{BETA-10}
\begin{split}
\beta_0&= \theta\left|{ f_{j+1/2}'}{\Delta x}-
\frac{23}{24}{f_{j+1/2}''' }\,{\Delta x}^3
\right| +\left|f_{j+1/2}''{\Delta x}^2
- \frac{3}{2} f_{j+1/2}'''{\Delta x}^3 \right|
+ \mO(\Delta x^{4}), \\
\beta_1 &={\theta} \left|  f_{j+1/2}' {\Delta x}
+ \frac{1}{24} f_{j+1/2}'''{\Delta x}^3\right|
+ \left|f_{j+1/2}''{\Delta x}^2
-\frac{1}{2} f_{j+1/2}''' {\Delta x}^3
\right| + \mO(\Delta x^{4}) , \\
\beta_2 & = \theta \left| f_{j+1/2}'{\Delta x}
+\frac{1}{24} f_{j+1/2}''' {\Delta x}^3\right|
+ \left| f_{j+1/2}'' {\Delta x}^2
+ \frac{1}{2} f_{j+1/2}'''{\Delta x}^3 \right|
+ \mO(\Delta x^{4}).
\end{split}
\end{align}

\begin{proposition}\label{LEM-OMEGA}
Let $d_k$, $k=0,1, 2$, be the optimal weights in \eqref{ID-W}.
Assume that $\epsilon = \epsilon({\Delta x})$ in the definition of $\alpha_k$
\eqref{ALPHA-0} is chosen as $\epsilon = \Delta x^\gamma$
with $0< \gamma \leq  4$.
If  $f$ is smooth around the global stencil $\mS_5$,
then the weights $\omega_k$ in \eqref{WEIGHT}
satisfy the following condition
$$ |\omega_k -d_k | =\mO(\Delta x^4)$$
even near the critical points.
\end{proposition}
\begin{proof}
Taking  the Taylor expansion of $f$ around $\xj$,
we can find that there exists a positive integer $r \in \NN$ such that
each $\beta_k$ in \eqref{BETA-10} can be expressed as
\ba\label{COND-B}
\beta_k = c |f_{j+1/2}^{(r)}\Delta x^r| +\mO(\Delta x^{r + 1})
\ea
with a constant $c>0$ independent of $f$ and $\Delta x$.
Certainly, if $\xj$ is not a critical point of  $f$,
then $r=1$.
Moreover, the truncation of the global smoothness indicator $\tau_5$
is of the form
\ba\label{COND-T}
\tau_5 = |f^{(4)}_{j+1/2} \Delta x^4| + \mO(\Delta x^5).
\ea
Then, we first consider the case $2r \leq \gamma$.
Substituting $\epsilon = \Delta x^\gamma$ in \eqref{ALPHA-0}
and by using \eqref{COND-B} and \eqref{COND-T},
it is straightforward that
\begin{align}\label{GB-10}
\begin{split}
\frac{\tau_5^2}{\beta_k^2 + \epsilon}
=\frac{\tau_5^2}{\beta_k^2 + \Delta x^\gamma}
= c_f \Delta x^{8 - 2r}  \frac{1 + \mO(\Delta x)}
{ 1+ \mO(\Delta x^{\gamma -2r})}
\end{split}
\end{align}
for some constant $c_f>0$.
By hypothesis, $0<\gamma \leq 4$ and $2r \leq \gamma$ so that
it yields the relation
\begin{align}\label{HY-10}
\alpha_k = d_k \Big (1+ \frac{\tau_5^2}{\beta_k^2 + \epsilon }\Big )
=d_k + \mO(\Delta x^4).
\end{align}
Further, since $d_1+d_2+d_3 =1$,
putting \eqref{HY-10}  into \eqref{WEIGHT} clearly verifies that
$|d_k- \omega_k| = \mO(\Delta x^4)$, regardless of the issue of
the critical points.
Also, in the case  $2r > \gamma$, it can be proved similarly.
Therefore, the proof is completed.
\end{proof}

\begin{table*}[t!]
\tabcolsep=18.pt
\caption{\label{tab:order-1d} $L^1$ and $L^\infty$ approximation errors
and orders of accuracy for the one-dimensional Euler equation
\eqref{euler2_1} at $t=4$.
}
\begin{tabular}{c c c c c}
\hline\hline
&   WENO-JS   &WENO-M       &WENO-Z       &WENO-H       \\
	\hline
$N$&  \multicolumn{4}{c}{$L^1$ approximation error (order)}  \\
	\hline
50& 3.98E-02 (~---~)&9.70E-03 (~---~)&9.62E-03 (~---~)&7.26E-03 (~---~)\\
100& 1.86E-03 (4.42)& 2.69E-04 (5.17)& 2.75E-04 (5.13)&  2.81E-05 (8.01)\\
200& 5.85E-05 (4.99)& 8.35E-06 (5.01)& 8.36E-06 (5.04)&  4.49E-07 (5.97)\\
400& 1.83E-06 (5.00)& 2.61E-07 (5.00)& 2.61E-07 (5.00)&  7.04E-09 (6.00)\\
800& 5.71E-08 (5.00)& 8.16E-09 (5.00)& 8.16E-09 (5.00)&  9.53E-11 (6.21)\\
\hline\hline
$N$&    \multicolumn{4}{c}{$L^\infty$ approximation error (order)}\\
\hline
50&6.03E-02 (~---~)&1.49E-02 (~---~)&1.49E-02 (~---~)& 1.11E-02 (~---~)\\
100& 2.71E-03 (4.47)&  4.18E-04 (5.16)& 4.49E-04 (5.05)& 4.81E-05 (7.85)\\
200& 9.81E-05 (4.79)&  1.31E-05 (5.00)& 1.34E-05 (5.07)& 7.08E-07 (6.08)\\
400& 3.28E-06 (4.90)&  4.10E-07 (5.00)& 4.12E-07 (5.02)& 1.11E-08 (6.00)\\
800& 1.03E-07 (5.00)&  1.28E-08 (5.00)& 1.28E-08 (5.00)& 1.50E-10 (6.20)\\
\hline
\end{tabular}

\vskip 0.2truein
\caption{\label{tab:order-2d}
$L^1$ and $L^\infty$ approximation errors and orders of accuracy
for the two-dimensional Euler equation \eqref{euler2_1} at $t=4$
}
\tabcolsep=15.pt
\begin{tabular}{c c c c c}
\hline\hline
&   WENO-JS   &WENO-M       &WENO-Z       &WENO-H      \\
\hline
$N \times N$ &\multicolumn{4}{c}{$L^1$ approximation error (order)}   \\
\hline
25$\times$25 & 3.06E-01 (~---~)& 2.70E-01 (~---~)&  2.26E-01 (~---~)& 2.03E-01 (~---~)  \\
50$\times$50 & 5.57E-02 (2.46)&  1.40E-02 (4.27)&  1.40E-02 (4.01)& 1.06E-02 (4.26)\\
100$\times$100 & 2.71E-03 (4.36)& 4.01E-04 (5.13)&  4.12E-04 (5.09)& 1.43E-05 (9.54)\\
200$\times$200 & 8.77E-05 (4.95)& 1.25E-05 (5.00)&  1.25E-05 (5.04)& 2.28E-07 (5.97)\\
400$\times$400 & 2.74E-06 (5.00)& 3.91E-07 (5.00)&  3.91E-07 (5.00)& 3.54E-09 (6.01)\\
\hline\hline
$N \times N$ &  \multicolumn{4}{c}{$L^\infty$ approximation error (order)}\\
\hline
25$\times$25  &  4.81E-01 (~---~)& 4.25E-01 (~---~)&  3.54E-01 (~---~)&   3.22E-01 (~---~)\\
50$\times$50 &  8.20E-02 (2.55)& 2.17E-02 (4.29)&  2.17E-02 (4.03)&    1.63E-02 (4.3)\\
100$\times$100 & 3.75E-03 (4.45)& 6.26E-04 (5.11)&  6.67E-04 (5.02)&    2.78E-05 (9.1)\\
200$\times$200 & 1.42E-04 (4.72)& 1.96E-05 (5.00)&  2.00E-05 (5.06)&    3.63E-07 (6.2)\\
400$\times$400 & 4.75E-06 (4.90)& 6.15E-07 (5.00)&  6.18E-07 (5.02)&    5.59E-09 (6.0)\\
\hline\hline
\end{tabular}
\end{table*}

\subsection{Accuracy test for smooth periodic Euler equations}
The goal of  this subsection is to demonstrate the convergence rate
of accuracy of the proposed WENO-H scheme.
We especially show that the carefully chosen exponential approximation
space can improve the accuracy of the WENO reconstruction.
The desired order of accuracy of WENO-H
is tested by solving the following Euler equation for one and two-dimensional
cases:
\begin{equation}\label{euler2_1}
        U_t + F(U)_x + G(U)_y = 0,
\end{equation}
with
\begin{align*}
U=
\begin{bmatrix}
\rho\\  \rho u \\  \rho v \\  E
\end{bmatrix}, \
F(U)=
\begin{bmatrix}
\rho u \\ p + \rho u^2  \\ \rho uv  \\ u(p+E)
\end{bmatrix}\quad  {\rm and} \quad
G(U)=
\begin{bmatrix}
\rho v\\ \rho vu \\  p + \rho v^2\\  v(p+E)
\end{bmatrix}.
\end{align*}

Here, $\rho, u, v$, and $E$ indicate the density,
particle velocities (along the  $x$ and $y$-directions),
and total energy, respectively.
The pressure $p$ has a relation with the total energy, that is,
ideal gas equation state:
$$ E= \frac{p}{\gamma-1}+ \frac{\rho (u^2 + v^2)}{2} $$
with $\gamma$ the ratio of specific heats.
Here, we set  $\gamma=1.4$.
The initial data is
\[
\rho(x,y,t)=1+0.5\sin(4\pi(x+y))
\]
with $u=1$, $v=-1/2$, and $p=1$.
The exact solution on the unit square is
\[
\rho(x,y,t)=1+0.5\sin(4\pi(x+y -t(u+v))),
\]
and the periodic boundary conditions are employed.
We perform the numerical simulation until the final time $t=4$.
For the time evolutions, we use non-TVD RK4 \cite{shu2} with
$\Delta t = \Delta x^{6/4}$.  The numerical results of WENO-H
and other well-known fifth-order WENO schemes
are presented in Table~\ref{tab:order-1d} and
Table~\ref{tab:order-2d} for one and two-dimensional problems respectively.
The $L^1$- and $L^\infty$-errors and convergence orders
of density $\rho$ are reported.
In addition, we also compare the effectiveness of these WENO schemes
by computing the CPU time versus $ L^{\infty}$-error using various grids.
In the comparison,  the $L^\infty$-errors against CPU time are presented
in Fig. \ref{fig:2d-efficiency} for one and two-dimensional problems.
Each marker indicates `CPU time-errors' at
$50 \times 2^k$ and $(25\cdot 2^k)\times (25\cdot 2^k)$ grid points
for one and two-dimensional cases with $k=0,1,2,3,4$.
The WENO-H scheme shows better efficiency compared to other WENO schemes.

\begin{figure*}
\begin{center} 
\includegraphics[width= 0.355\textwidth]{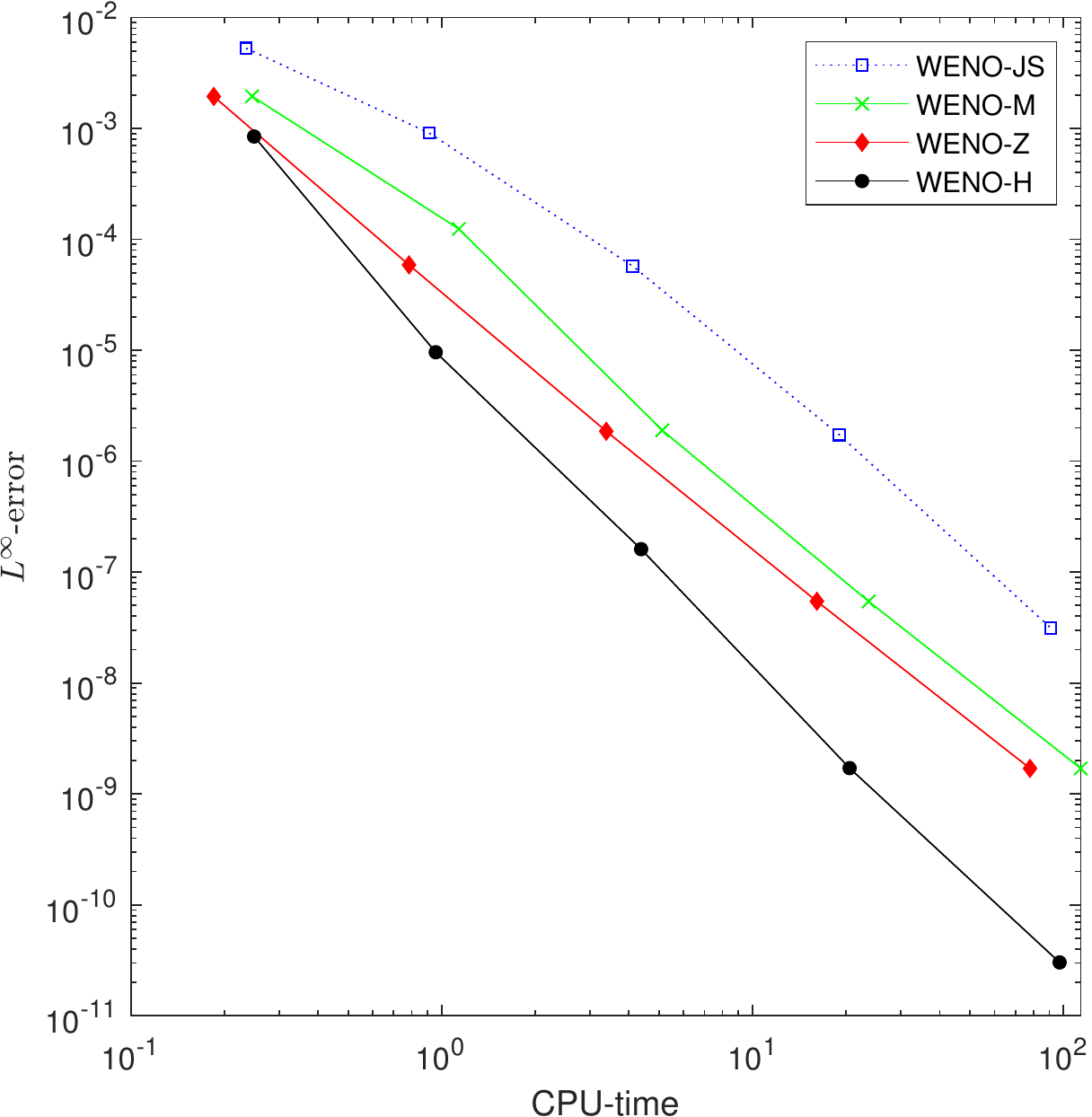}
\includegraphics[width= 0.35\textwidth]{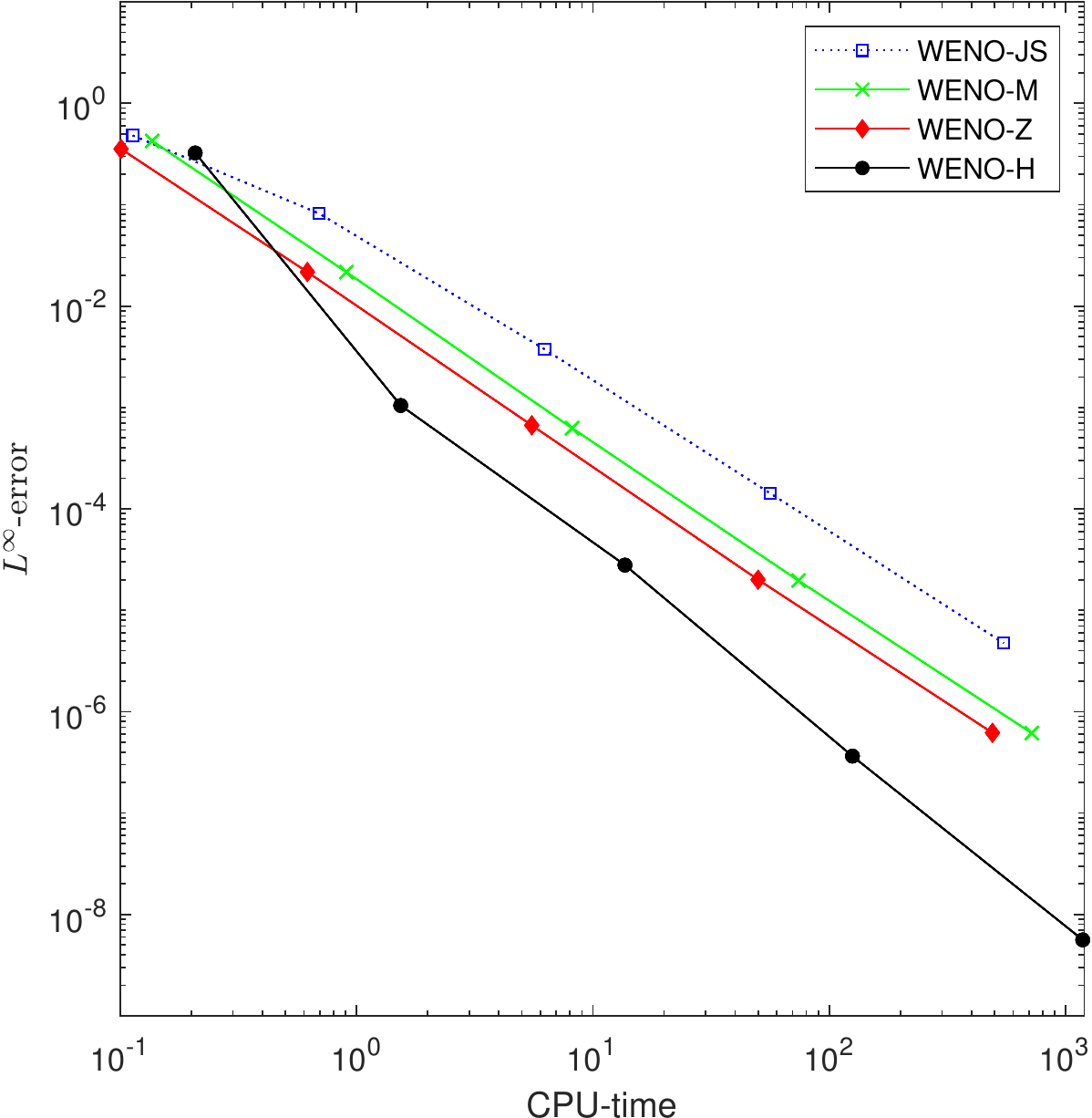}
\end{center} \caption{\label{fig:2d-efficiency}
Numerical Efficiency (CPU time versus errors) for the one (left)
and two-dimensional (right) Euler equations \eqref{euler2_1}.
}
\end{figure*}

\section{Numerical Results} \label{NUM-REL}

In this section, we provide some experimental results to illustrate
the performance of the WENO-H scheme.
The experimental results of the WENO-H scheme are compared with
those of other well-known fifth-order WENO schemes:
WENO-JS, WENO-M and WENO-Z.
For the evaluation of the shock capturing abilities of the proposed algorithm,
the simulations are performed for several benchmarks of
one and two-dimensional scalar and system of conservation laws.
For all the numerical experiments in this section,
we employ the third-order TVD Runge-Kutta-type discretization
for time evolution.

\subsection{Scalar Test Problems}

We investigate the behavior of the WENO-H
method for the one-dimensional advection equation
with an initial data including unusual edges and contact discontinuities.

\begin{example}\label{LE-0}
{\rm (Linear equation) Let us solve the advection equation:
\begin{equation}\label{ex0}
	q_t + q_x = 0, \quad t \in \Bbb{R}^+
\end{equation}
with the initial condition specified as
\begin{equation} \label{SING}
q(x,0) = q_0(x) =
\begin{cases}
-x \sin(\frac{3 \pi}{2}x^2) & \text{for $ x \in [-1, -\frac{1}{3}]$}, \\
|\sin(2\pi x) | & \text{for $ x \in (-\frac{1}{3}, \frac{1}{3}]$}, \\
2x - 1 -\frac{1}{6} \sin(3\pi x) & \text{for $ x \in (\frac{1}{3}, 1]$}.
\end{cases}
\end{equation}
We set the periodic boundary conditions and carry out the computation
until the final time $t = 11$ with $\Delta x = 0.01$.
The CFL condition number is $0.4$.
The numerical results of this advection equation with initial
condition \eqref{SING} are shown in Fig. \ref{fig_sing}.
We observe that the WENO-H method has smaller errors than
other WENO fifth-order schemes near the singular points.
}
\end{example}

\begin{figure*}[t!]
\begin{center}
\includegraphics[width= 0.70\textwidth]{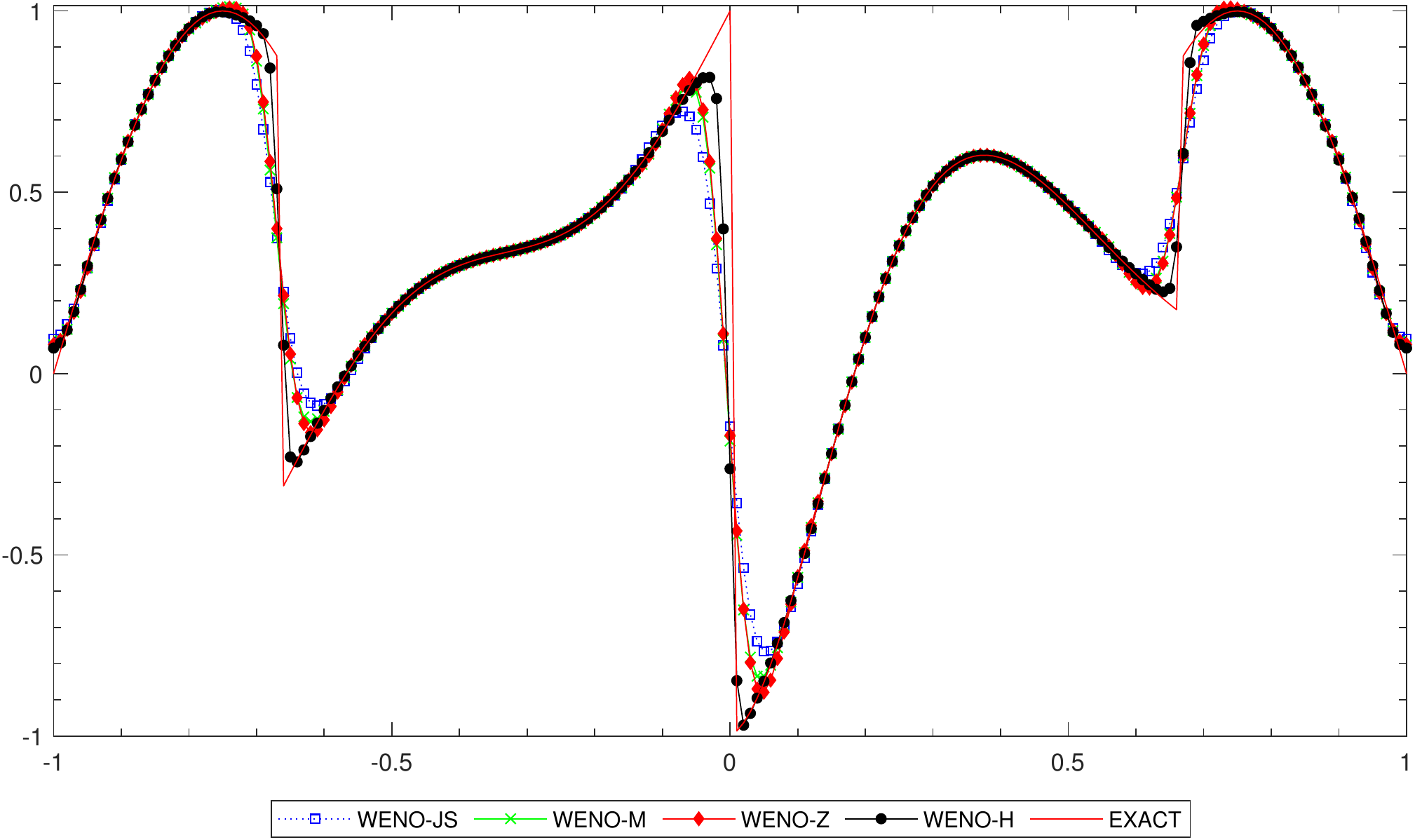}
\end{center} \vskip -.1truein
\caption{\label{fig_sing} Comparison of the analytic solution with
the numerical results of the advection equation with initial
conditions \eqref{SING} with WENO-JS, WENO-M, WENO-Z, and WENO-H at $t=11$
with $200$ grid points.
	}
\end{figure*}

\subsection{One-dimensional Euler Systems}

Let us consider the one-dimensional Euler gas dynamics for ideal gases.
The characteristic decomposition is performed to generalize the WENO methods
\cite{shu3}.

\begin{figure*}[t!]
\begin{center} 
\includegraphics[width=0.85\textwidth, height=0.27\textheight]{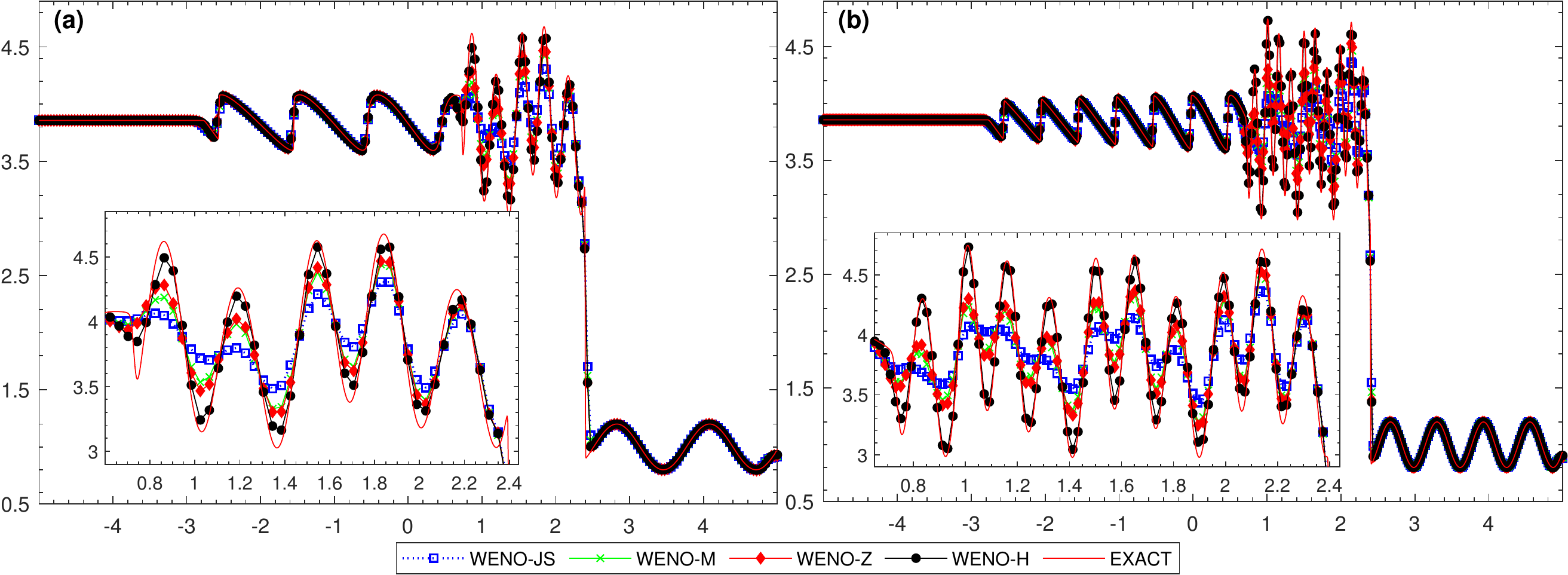}
\end{center} 
\caption{\label{ent_5} Density profiles of the shock-entropy interacting
of Shu-Osher \cite{shu2} by WENO-JS, WENO-M, WENO-Z, and WENO-H.
(a) $t=1.8$ with 250 grid points for $k=5$, (b) $t=1.8$ with 500 grid points
for $k=10$.
}

\vskip0.1truein
\begin{center} 
\includegraphics[width=.85 \textwidth,height=0.27\textheight]
  {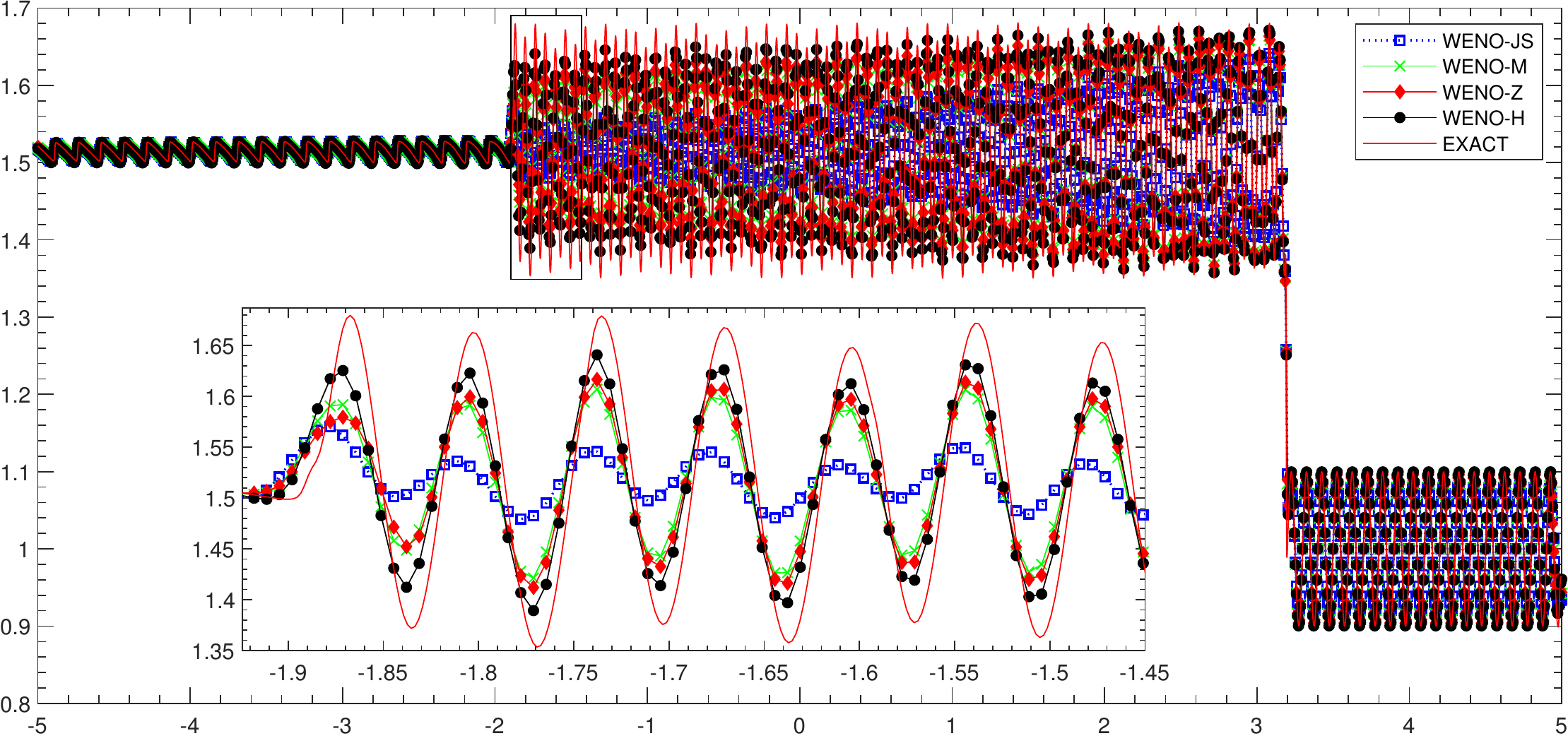}
\end{center} \caption{\label{toro1}
Numerical results with WENO-JS, WENO-M, WENO-Z, and WENO-H at $t=5$
with $1500$ grid points.
	}
\end{figure*}
\begin{example}
{\rm
We apply the WENO-H scheme to the shock-density wave interaction test problem
that describes shock interacting with entropy waves.
This model problem was introduced by Shu and Osher \cite{shu2}
to test the capability of a high-order
WENO scheme to capture the high frequency waves.
The solution of this example includes large scale waves, small shocks and
fine scale structures.
We solve this problem on the interval $[-5,5]$ with the
specified initial condition:
\begin{equation*}
(\rho, u, p) =
\begin{cases}
(3.857143,2.629369,10.33333) & \text{for $ x \in [-5, -4)$}, \\
(1+ \varepsilon \sin(kx),0,1) & \text{for $ x \in [-4, 5]$}
\end{cases}
\end{equation*}
where $\varepsilon=0.2$ is the amplitude of the entropy wave and
$k$ is wave number of the entropy wave.
A shock wave flowing to the right (with speed `Mach 3') interacts sine wave
in a perturbed density disturbance such that it yields a flow field with
discontinuities as well as smooth structures.
We simulate this problem for $k=5, 10$ until the output time $t=1.8$
using the CFL number $0.5$.
The exact solution of this model problem is unknown.
So, the reference solution is computed
by the classical fifth-order WENO-JS scheme with $3200$ points.
Fig. \ref{ent_5} plots a comparison of the densities $\rho$ for all schemes
at time $t = 1.8$.
Notice that WENO-H resolves most of the waves with a good accuracy
(to their amplitudes) over other tested methods.

In addition, as a variation of the Shu-Osher problem,
let us solve Titarev-Toro problem
with the initial condition given as follows \cite{T-TORO}:
\begin{equation*}
(\rho, u, p) =
\begin{cases}
(1.515695, 0.523346, 1.80500) & \text{for $ x \in [-5, -4)$}, \\
(1+ 0.1 \sin(20 \pi x), 0, 1) & \text{for $ x \in [-4, 5]$}.
\end{cases}
\end{equation*}
The simulation is performed up to time $t=5$ with $\Delta x=1/150$.
Fig. \ref{toro1} shows the numerical solutions  on a grid with
$1500$ grid points (i.e., $\Delta x=1/150$)
for all the computed  WENO schemes.
We observe that the oscillatory wave pattern behind shock entropy
wave interactions is well captured by WENO-H better than other WENO methods.
}
\end{example}

\begin{figure*}[t!]
\begin{center}
\includegraphics[width= 0.8\textwidth, height=0.25\textheight]{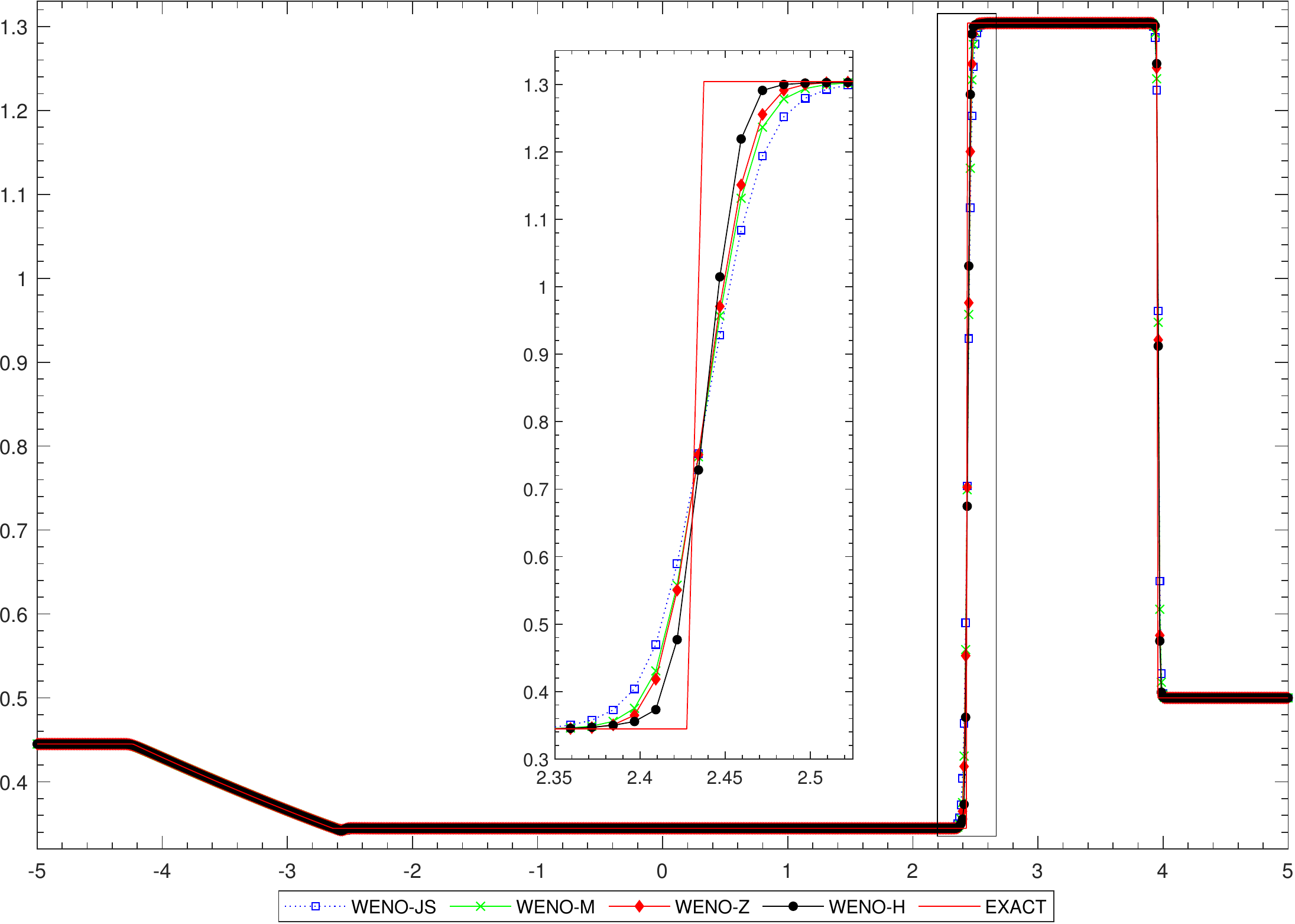}
\end{center}
\caption{\label{lax} Numerical results of Lax problem \cite{lax1}
with WENO-JS, WENO-M, WENO-Z, and WENO-H at $t=1.6$ with $200$ grid points.
}

\vskip 0.2truein
\begin{center} 
\includegraphics[width=0.8\textwidth, height=0.25\textheight]{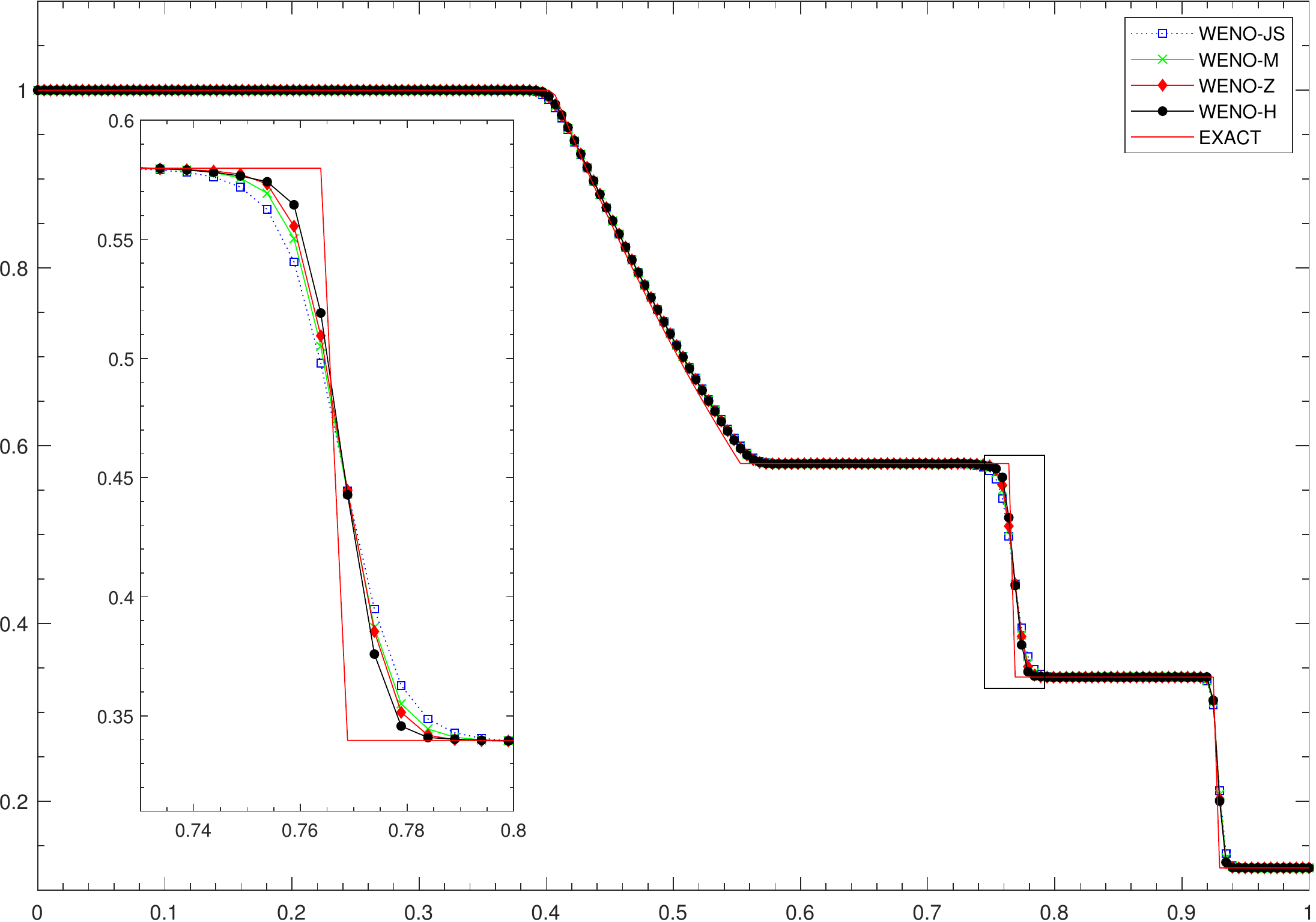}
\end{center} \caption{\label{sod}
Numerical results of Sod problem \cite{sod} with WENO-JS, WENO-M, WENO-Z,
and WENO-H at $t=0.2$ with $200$ grid points.
}
\end{figure*}

\begin{example}
{\rm We test one-dimensional Euler equation for the Lax problem \cite{lax1}.
The initial condition is specified by
\begin{equation*}
	(\rho, u, p) =
	\begin{cases}
	(.445, .698,3.528) & \text{for $ x \in [-5, 0)$}, \\
	(.5, 0, .571) & \text{for $ x \in [0, 5]$}
	\end{cases}
\end{equation*}
with $\gamma=1.4$.
The computation is performed up to time $t= 0.16$ with $200$ grid
points (i.e.,  $\Delta x = 1/20$).
Fig. \ref{lax} presents the exact solution (reported in Toro \cite{toro1})
and the density $\rho$ profiles obtained by several  WENO schemes.
The result of WENO-H is closer to the exact solution
and captures the shock and contact transitions nearby discontinuities
better than other WENO schemes.
	}
\end{example}

\begin{example}
{\rm  In this example, we solve the one-dimensional Euler equation
for the Sod problem
\cite{sod} with the Riemann initial condition given by
\begin{equation*}
	(\rho, u, p) =
	\begin{cases}
	(1,0.75,1) &  \text{for $ x \in [0, 0.5)$}, \\
	(0.125,0,0.1) & \text{for $ x \in [0.5, 1]$}
	\end{cases}
\end{equation*}
with $\gamma=1.4$.
The computation has been performed up to time $t = 0.2$.
The computed density distributions and exact solution are shown
in Fig. \ref{sod} with 200 grid points (i.e. $\Delta x = 1/200$).
The exact solution is obtained by using the exact Riemann solver \cite{toro1}.
One can see that the solution of WENO-H well captures the shock and contact
discontinuity without redundant oscillations better than
WENO-JS, WENO-M and WENO-Z do.
}
\end{example}

\subsection{Two-dimensional Euler Systems}
The numerical results of two-dimensional compressible Euler equations are
provided in this section.
We specify an initial condition for each test problem and set
$\gamma = 1.4$ except the two-dimensional Rayleigh-Taylor instability problem.

\begin{figure*}[ttp]
\begin{center}
\includegraphics[width=1.10\textwidth]{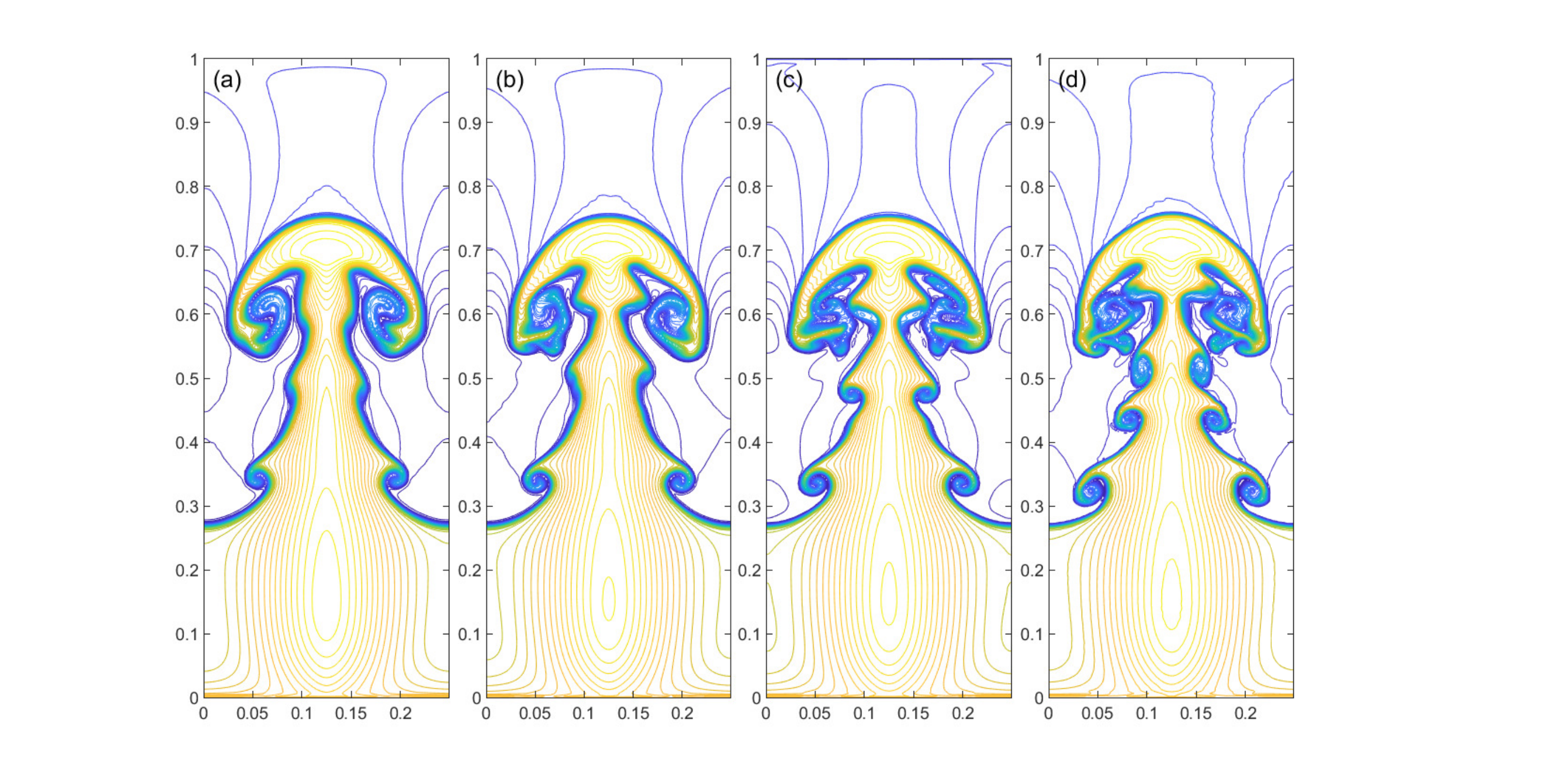}
\end{center}
\vskip -0.2truein
\caption{\label{rti_2} Two-dimensional Rayleigh-Taylor instability
\cite{Shi, Xu}:
(a) WENO-JS, (b) WENO-M, (c) WENO-Z, and (d) WENO-H at $t = 1.95$
with $120 \times 480$ grid points.
	}
\end{figure*}

\begin{figure*}[t!]
\begin{center}
\includegraphics[trim=30 130 30 0,clip, width= 0.95\textwidth]
{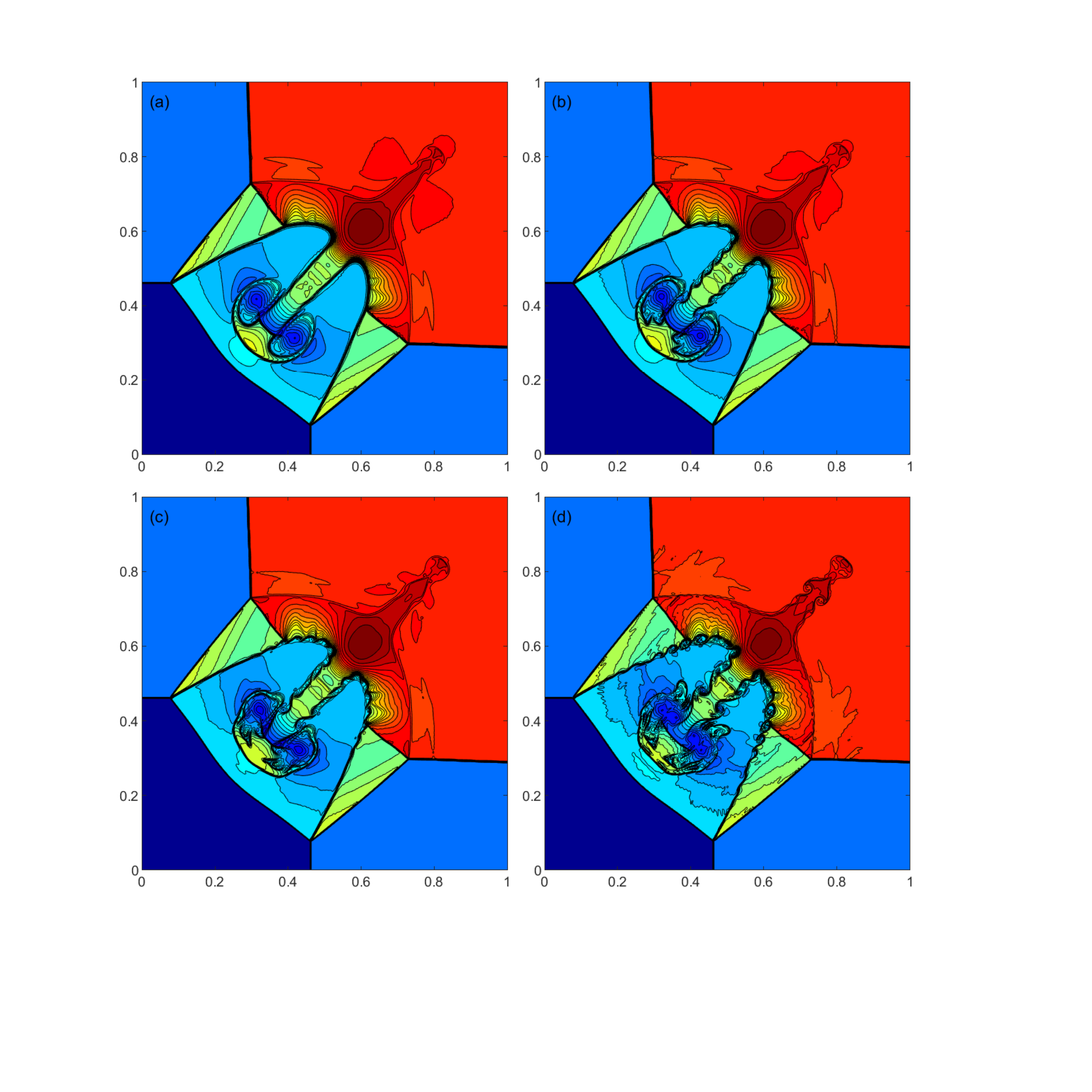}
\end{center}
\caption{\label{quad2_2} {Density profiles
of two-dimensional problem \cite{car}:
(a) WENO-JS, (b) WENO-M, (c) WENO-Z, (d) WENO-H at $t = 0.8$
with $500 \times 500$ grid points.}
	}
\end{figure*}

\begin{example}
{\rm (Two-dimensional Rayleigh-Taylor instability)
This model problem describes the interface instability between fluids with different
densities, where the heavy fluid moves down to the light fluid.
This problem has been computed to check the numerical dissipation
(e.g., \cite{Shi, Xu}).
In this example,
the simulation is performed on the domain $[0,0.25] \times [0,1]$
with the initial condition specified by	
\begin{equation*}
(\rho, u , v,  p) =
\begin{cases}
(2, 0, -0.025 \displaystyle\sqrt{\frac{5 p}{3 \rho}} \cos(8 \pi x), 2y+1) & \text{for $ y \in [0, 0.5)$}, \\
(1, 0, -0.025 \displaystyle\sqrt{\frac{5 p}{3 \rho}} \cos(8 \pi x), 2y+ 1.5) & \text{for $ y \in [0.5, 1]$}.
\end{cases}
\end{equation*}
The gravitational effect can be obtained by adding $ \rho $ and $ \rho v $
to the right of $y$-momentum and the energy equation respectively.
We set the ratio of specific heats as $\gamma = 5/3$.
The right and left-hand  boundaries are taken by the reflective boundary
conditions.  The velocity is 0, and we set
$(\rho, p) = (1, 2.5)$ for the top boundary condition
and $(\rho, p) = (2, 1)$ for the bottom boundary condition.
The results are simulated up to time $t=1.95$.
Fig. \ref{rti_2} depicts the density contour lines of the
solutions computed by the WENO-H and other
fifth-order WENO schemes with $120 \times  480$ grid points.
The appearance of the small structure in the flow is a measure
of the small magnitude of the intrinsic numerical viscosity of
the numerical schemes.
We can observe that the WENO-H scheme is able to capture complex structures
better than other schemes and improves significantly the contact discontinuity
resolution.

}
\end{example}

\begin{figure*}[t!]
\begin{center}
\includegraphics[trim=40 130 30 0,clip,  width=0.90\textwidth]
{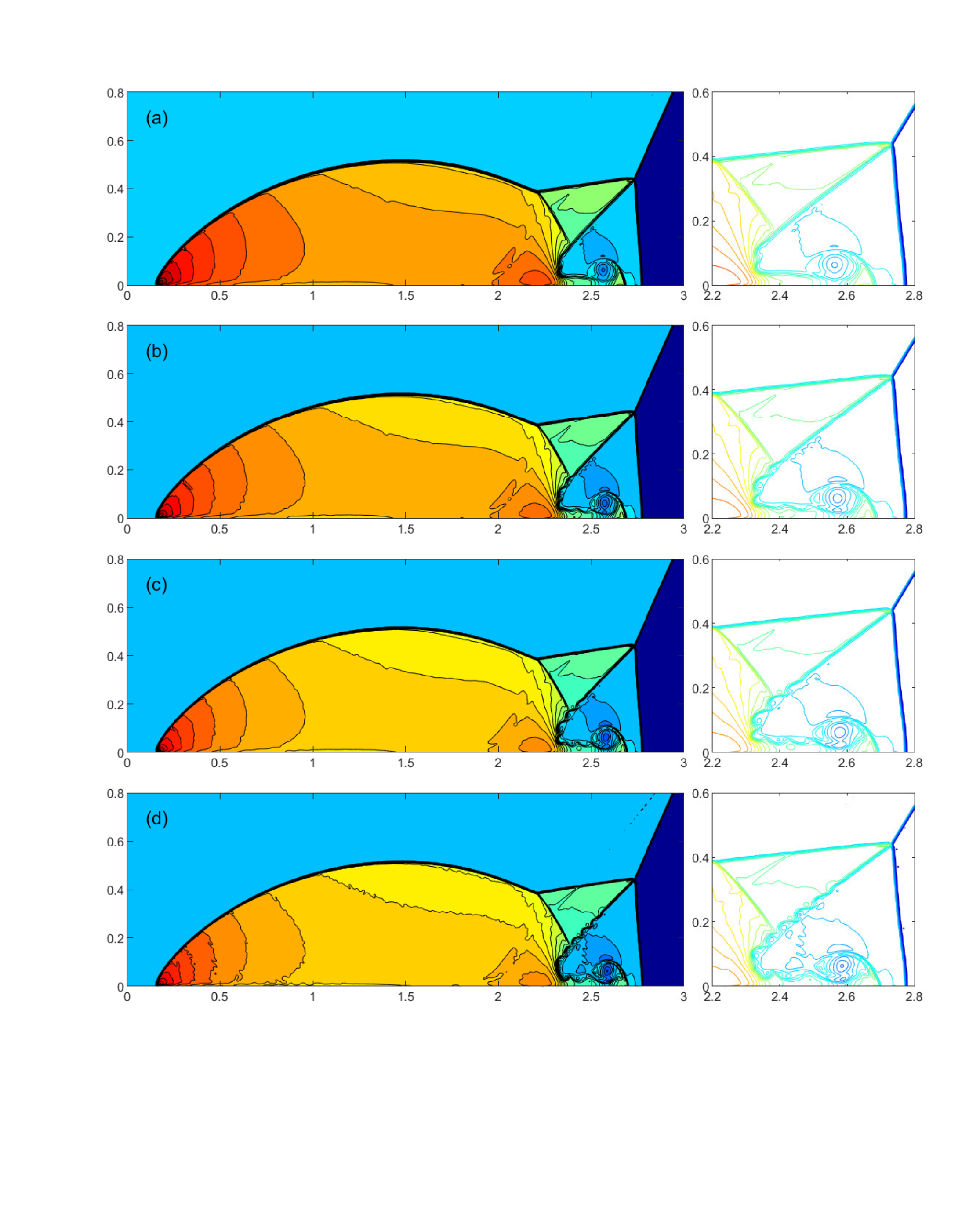}
\end{center}
\caption{\label{double_mach_z1} {Double Mach reflection of a strong
shock \cite{wood}: (a)  WENO-JS, (b) WENO-M, (c) WENO-Z, (d) WENO-H
at $t = 0.2$ with $960 \times 240$.}
}
\end{figure*}

\begin{figure*}[t!]
\begin{center}
\includegraphics[trim=30 100 30 0,clip, width=0.95\textwidth]{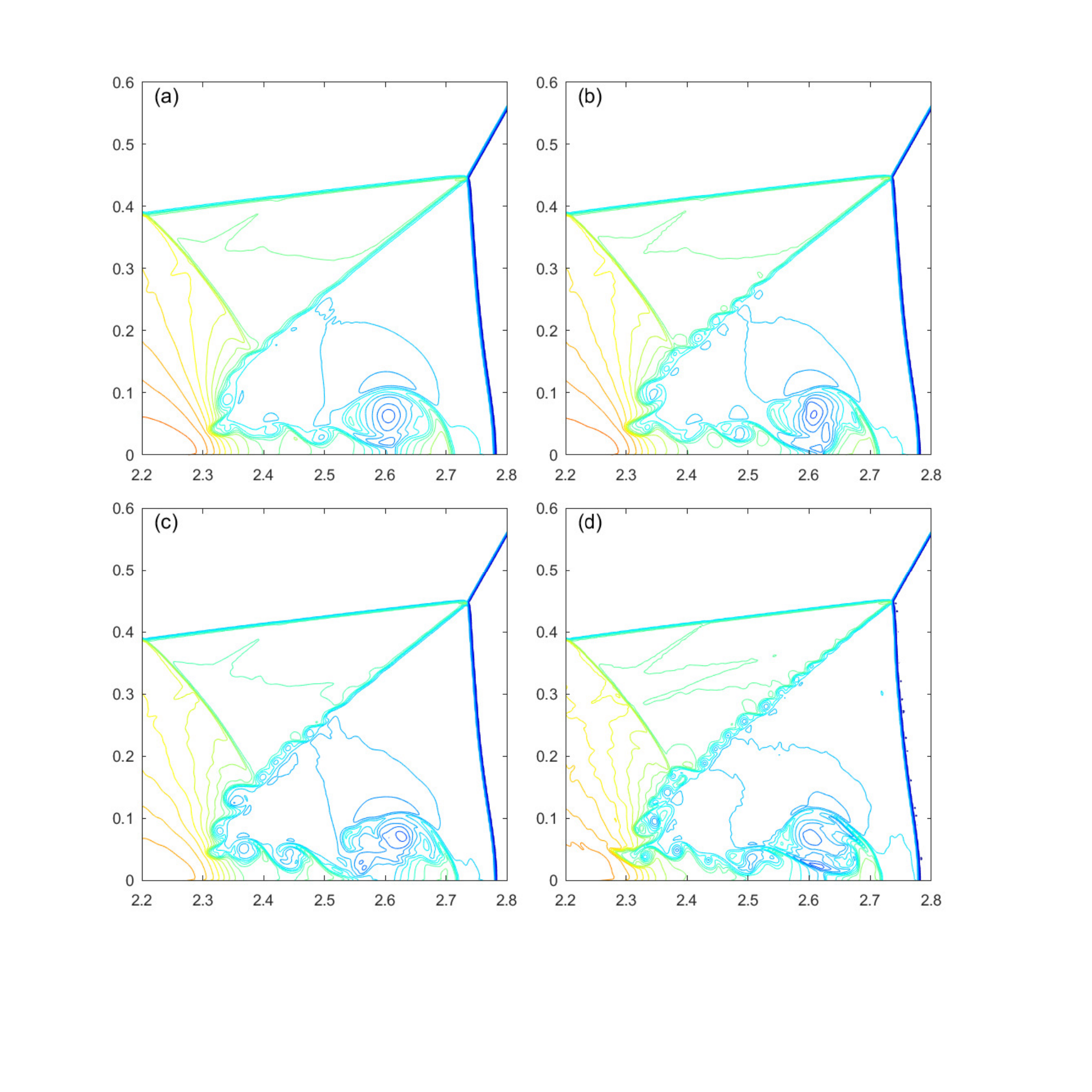}
\end{center}
\caption{\label{double_mach_z2} Double Mach reflection of a strong
shock \cite{wood}: (a) WENO-JS, (b) WENO-M, (c) WENO-Z, and (d) WENO-H
at $t = 0.2$ with $1920 \times 480$ grid points.
}
\end{figure*}

\begin{example}
{\rm (Two-Dimensional Riemann Problem for Gas Dynamics)
We consider the third configuration of the two-dimensional
Riemann problems for gas dynamics \cite{car}.
The computational domain is $[0,1] \times [0,1]$
which is divided into $4$ quadrants by lines $x=0.8$ and $y=0.8$.
In each quadrant, the initial data is set as constant:
\begin{equation*}
(\rho, u , v,  p) =
\begin{cases}
(1.5, 0, 0, 1.5) & \text{for $ (x,y) \in [0.8, 1] \times [0.8, 1]$}, \\
(0.5323, 1.206, 0, 0.3)& \text{for $ (x,y) \in [0, 0.8] \times [0.8, 1]$}, \\
(0.138, 1.206, 1.206, 0.029) &\text{for $(x,y)\in [0, 0.8]\times [0, 0.8]$},\\
(0.5323, 0, 1.206, 0.3) & \text{for $ (x,y) \in [0.8, 1] \times [0, 0.8]$}
\end{cases}
\end{equation*}
\noindent
with outflow boundary conditions.
The computation is carried out until time $t=0.8$ with $500 \times 500$
grid points.
The performance of WENO-H is compared with those of other
WENO schemes in Figs. \ref{quad2_2}.


%
}
\end{example}


\begin{figure*}[t!]
\centering
\includegraphics[width=0.85\textwidth]{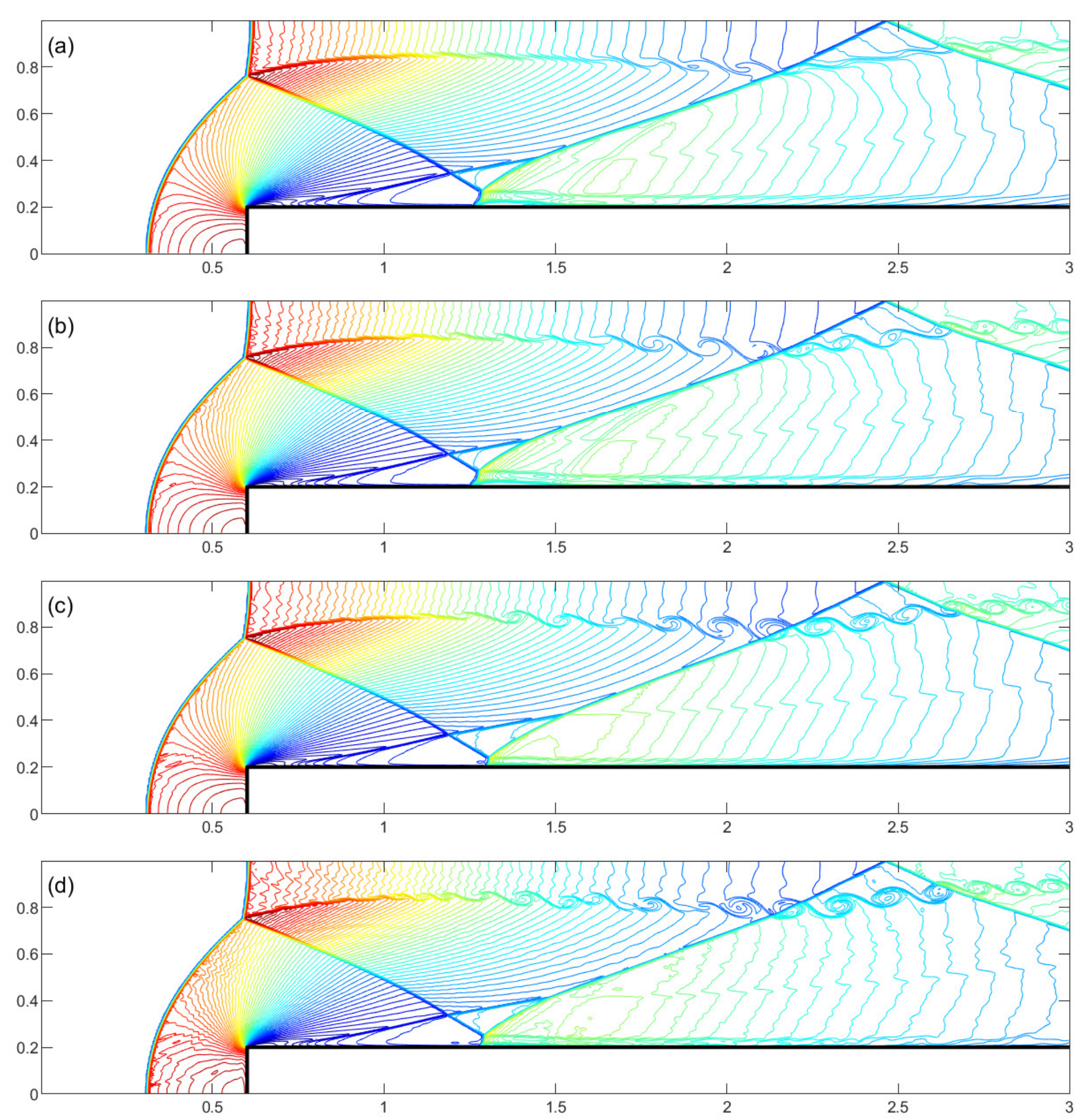}\hskip2cm
\caption{\label{wind_tunnel} Density profiles of Mach 3 Wind tunnel
with a step \cite{wood}:
(a) WENO-JS, (b) WENO-M, (c) WENO-Z, and (d) WENO-H
at $t=4$ with $768\times 256$ grid points.  }
\label{fig:wind_tunnel}
\end{figure*}

\begin{example}
{\rm (Double Mach reflection of a strong shock)
This model problem was introduced by Woodward and Colella \cite{wood}.
Since then, it has been used to test the capability of
a high accurate scheme to capture small-scale structures and shocks.
We test this problem on the domain $[0, 4] \times [0, 1]$.
This example is initialized with a right-moving Mach 10 oblique shock
oriented at an angle of $60^{\circ}$ to the horizontal axis passing
through the point $(x,y)=(\frac{1}{6},0)$.
Exact post-shock condition is used for the boundary conditions on
$x \in [0, \frac{1}{6}]$ and the rest part of the bottom is used
as a reflective boundary condition.
Left and right boundaries use inflow and outflow boundary conditions.
Exact motions of the Mach $10$ shock are used to the boundary of top parts.
Density $\rho = 1.4$ and pressure $p=1$ are set for the unshocked fluid.
The problem was run till $t = 0.2$.
Fig. \ref{double_mach_z1} and \ref{double_mach_z2}
plot the  density profiles computed with the WENO-H and WENO-JS, WENO-M and
WENO-Z schemes with $960 \times 240$ and $1920 \times 480 $ grid
points respectively.
We can see that the WENO-H scheme yields better resolutions
than other WENO methods.
}
\end{example}

\begin{figure*}[h!]
\begin{center} 
\includegraphics[width=0.8\textwidth]{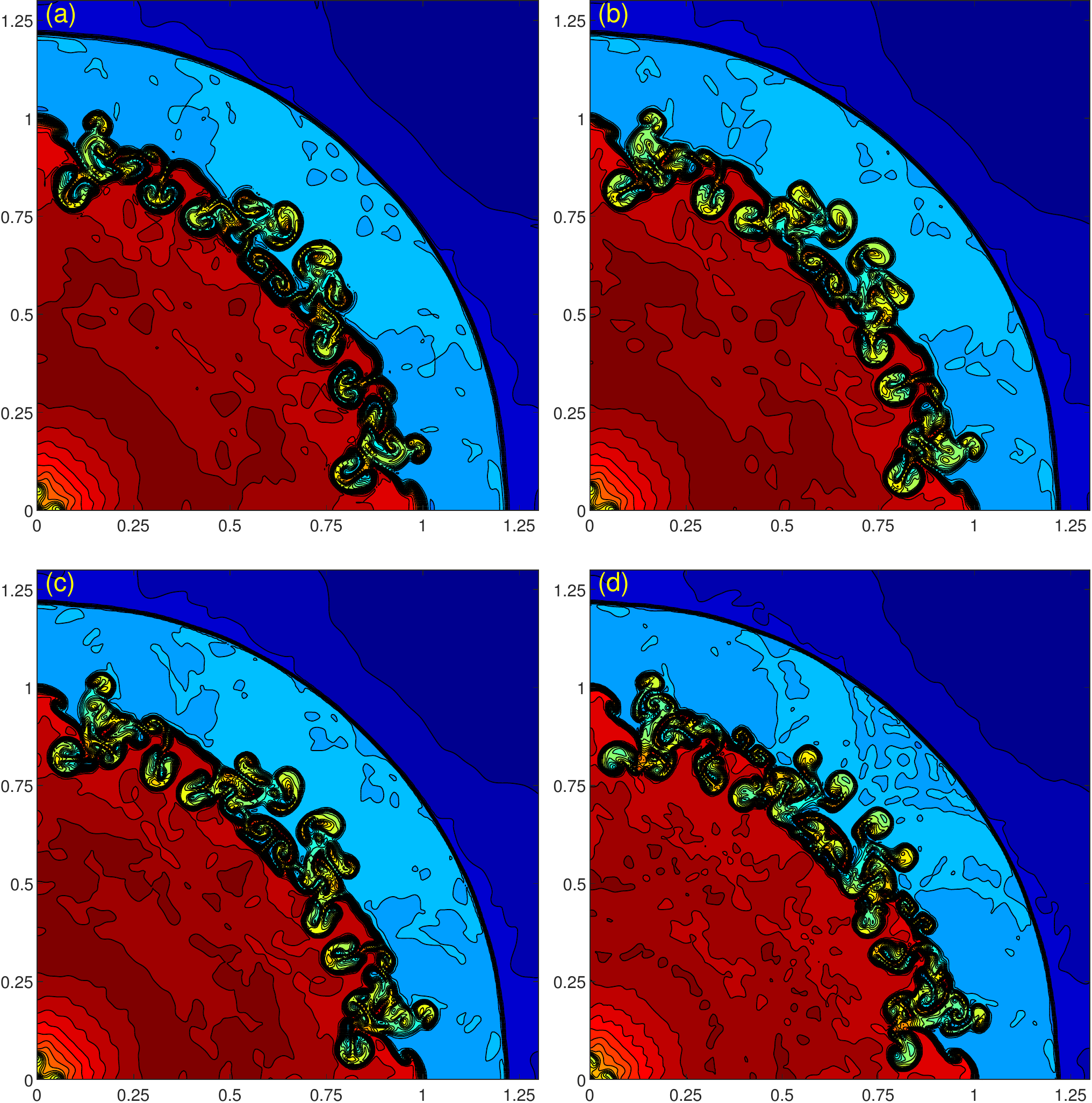}
\end{center} \caption{\label{expl2} Density profiles of the first quadrant for
Explosion \cite{toro1}, (a) WENO-JS, (b) WENO-M, (c) WENO-Z, and (d) WENO-H
at $t=3.2$. }
\end{figure*}

\begin{example}
{\rm (A Mach 3 Wind Tunnel with a Step)	
This problem describes a Mach 3 flow with a forward-facing step
in a wind tunnel. It was first described by Emery \cite{Emery}
to compare several hydrodynamical methods.
Later, Woodward and Colella \cite{wood} used it to compare
several advanced numerical schemes.
We compute this problem in a wind tunnel with
one length unit width and three length units long.
The step is $0.2$ length units high and is located
$0.6$ length units from the left-hand end of the tunnel.
The reflective boundary conditions is assumed along the walls of the tunnel.
We also assume that the tunnel has an infinite width along
the direction orthogonal
to the calculation plane.
A gas is continuously supplied at the left boundary
with the pressure $1$, density $1$  and velocity $3$ respectively.
The corner of the step is the singularity of the flow, since it is the center point of the rarefaction fan.
After the bow shock is reflected in the step, the shock gradually reaches the top reflective wall of the tunnel around $ t=.65$.
Due to the reflections and interactions of the shocks, a triple point is formed,
from which the trail of vortices moves towards the right boundary.
Fig. \ref{wind_tunnel} plots the density profiles obtained by
WENO-H with the other WENO schemes
at the final time $t=4$ with $768 \times 256$ mesh grids.
We see that the roll-up of the vortex sheet is more clearly visible
with WENO-H.
}
\end{example}		

\begin{example}
{\rm (Explosion) We compute the explosion problem proposed in \cite{toro1}
(see also \cite{Liska}) which is a circularly symmetric two-dimensional
problem with initial circular region of higher density and pressure.
The circle is centered at the origin with radius 0.4.
The computation is performed on the domain $[-1.5, 1.5] \times [-1.5, 1.5]$
with the initial condition given by
\begin{equation*}
(\rho, u , v,  p) =
\begin{cases}
(1.000, 0, 0 , 1.0) & \text{if $x^2 + y^2 < 0.16$}, \\
(0.125, 0, 0 ,0.1)& \text{otherwise}
\end{cases}
\end{equation*}
with $\gamma = 1.4$.
We compute the solution until time $t=3.2$ with $600 \times 600$ mesh grids.
Fig. \ref{expl2} shows the density profiles obtained
by the four tested WENO schemes.
We can see that the numerical results by WENO-H are much `curlier'
at the contact surface than the results obtained by other tested methods.
This explains that WENO-H has substantially smaller dissipation
than other WENO schemes.
	}
\end{example}

\section{Conclusion}
In this paper, we have proposed an improved WENO schemes (called WENO-H)
for the numerical solution of the hyperbolic conservation laws.
The interpolation method is based on the space of exponential polynomials
with a tension parameter.  We proposed a practical approach
to determine the parameter of the exponential approximation
space  by taking into account the local data feature.
As a result, the proposed WENO scheme attains an improved order
of accuracy (that is, sixth-order) better than other fifth-order WENO methods
without loss of accuracy at critical points.
A detailed analysis is provided to verify the improved accuracy.
Further, modified nonlinear weights based on $L^1$-norm approach
were proposed along with a new global smoothness indicator.
The proposed WENO scheme resolve discontinuities sharply while
reducing numerical dissipation significantly.
Several experimental results of the WENO-H scheme
for the advection equation and the system of the Euler equations
are compared with those of the other fifth-order WENO scheme
to confirm the reliability of the method.
In the near future we generalized our approach to sixth or higher-order
WENO schemes.


%
\end{document}